\documentclass[11pt]{amsart}

\title{Hyperplane Incidences and Distance Sets in Higher Dimensions}
\author{Marianna Cs\"ornyei} 
\address{Department of Mathematics, University of Chicago, Chicago, IL 60637}  
\email{csornyei@math.uchicago.edu}

\author{D. M. Stull}
\address{Department of Mathematics, University of Chicago, Chicago, IL 60637}
\email{dmstull@uchicago.edu}

\usepackage{amsmath,amssymb,amsthm,amsfonts,mathtools,bm,xcolor}
\usepackage[T1]{fontenc}
\usepackage[colorlinks=true]{hyperref}

\theoremstyle{plain}
\newtheorem{theorem}{Theorem}[section]

\newtheorem*{claim*}{Claim}

\newtheorem{proposition}[theorem]{Proposition}
\newtheorem*{proposition*}{Proposition}

\newtheorem*{fact*}{Fact}

\newtheorem*{conjecture*}{Conjecture}
\newtheorem{corollary}[theorem]{Corollary}
\newtheorem{lemma}[theorem]{Lemma}
\newtheorem*{lemma*}{Lemma}

\newtheorem*{question*}{Question}
\theoremstyle{definition}\newtheorem{remark}[theorem]{Remark}
\theoremstyle{definition}\newtheorem*{remark*}{Remark}
\theoremstyle{definition}\newtheorem{definition}[theorem]{Definition}
\theoremstyle{definition}\newtheorem*{definition*}{Definition}
\theoremstyle{definition}
\theoremstyle{definition}
\theoremstyle{definition}
\theoremstyle{definition}\newtheorem*{example*}{Example}

\newcommand\lessim{\lesssim}
\renewcommand{\phi}{\varphi} 
\renewcommand{\epsilon}{\varepsilon} 
\def\eps{{\epsilon}}

\def\N{\mathbb{N}} 
\def\Q{\mathbb{Q}} 
\def\R{\mathbb{R}} 
\def\P{\mathbb{P}} 

\def\S{\mathcal{S}}
\def\P{\mathcal{P}}
\def\T{\mathcal{T}}
\def\D{\mathcal{D}}
\def\H{\mathcal{H}}
\begin{document}

\begin{abstract}
    We generalize Ren and Wang's incidence bound between points and lines in $\R^2$ \cite{RenWan23} to higher dimensions. We show how to use this incidence bound to improve the best known bound for Falconer's distance set problem in $\R^3$ and in $\R^4$. We show that if $d=3$ or $d=4$, and $E\subset \R^d$ is a Borel set of dimension $\dim_H(E) > d/2$, then
    \begin{equation*}
        \sup_{x\in E} \dim_H(\Delta_x(E)) \geq 2/3,
    \end{equation*}
    where $\Delta_x(E)$ is the pinned distance set of $E$ with respect to $x$.
    
    We also show how the incidence bound can be used to generalize the planar Furstenberg set bound, to sets in $\R^d$ that contain a $t$-dimensional set of hyperplanes, each of which contains an $s$-dimensional set of points, for any $d\ge 2$, $s \in (d-2, d-1]$ and $t \in (0, d]$.

\end{abstract}
\maketitle

\section{Introduction}
In this paper, we prove lower bounds on the discretized incidences between hyperplanes and points in $\R^d$. For any dyadic scale $\delta >0$, a $\delta$-hyperplane is the dual set of the points covered by a dyadic cube. We are interested in the following problem. Let $\P$ be a finite set of $\delta$-cubes in $\R^d$, and let $\H$ be a finite set of $\delta$-hyperplanes. Suppose that, for each cube $Q$ in $\P$, the number of $H \in \H$ which intersect $Q$ is at least $M$. How large must $\H$ be? 

To be nontrivial, we need additional assumptions about how `separated' the $\delta$-hyperplanes are. We will use a Frostman-type non-concentration condition to prove our main theorem. This non-concentration condition is the natural analog of the non-concentration condition used by Ren and Wang \cite{RenWan23} in their solution of the Furstenberg problem in the plane. We defer the precise definitions to Section \ref{ssec:prelimDiscretized}.  
\begin{theorem}\label{thm:hyperplanesIncidence}
    Let $d \geq 2$. For every $\epsilon > 0$, $s\in (d-2, d-1]$ and $t\in(0,d]$, there exists $\eta > 0$ such that for every small enough $\delta>0$ and every $M\in\N$, the following holds. Let  $\mathcal{P}\subseteq \D_\delta([0,1]^d)$ be a $(\delta, t, \delta^{-\eta})$-set and suppose that, for each $Q\in\P$, there is a $(\delta, s, \delta^{-\eta})$-set of $\delta$-hyperplanes $H(Q)$ such that each hyperplane in $H(Q)$ intersects $Q$ and $\vert H(Q) \vert \sim M$. Then for $\H := \bigcup_{Q\in\P} H(Q)$, we have
    \begin{equation}\label{eq:hyperplanesIncidenceBound}
        \vert \H\vert \gtrsim_\eps \delta^{-\min\{t, \frac{s+t - (d-2)}{2}, 1\} + \eps}M.
    \end{equation}
    
\end{theorem}
Theorem \ref{thm:hyperplanesIncidence} generalizes Ren and Wang's recent breakthrough on the Furstenberg set problem in the plane \cite{RenWan23}.

The proof of Theorem \ref{thm:hyperplanesIncidence} uses a discretized analog of the induction scheme developed in \cite{ChoCsoLutLutMayStu24} for estimating dimensions of exceptional sets of projections. However, the discretization of this induction requires new ideas, and, in particular, makes use of algorithmic methods. Briefly, the main idea of the induction is to use Ren and Wang's incidence estimate for tubes as a base case. We then use discretized analogs of Marstrand's projection theorem and Mattila's slicing theorem to reduce higher dimensional incidence estimates to lower dimensional incidence estimates. 

In Section \ref{sec:discretizedSlicingProjection} we prove our discretized analogs of Marstrand's projection theorem and Mattila's slicing theorem. We believe that these discrete results are interesting by themselves, and may be of future use.

\medskip
In Section \ref{sec:IncidenceHyperplanes} we use our discretized slicing and projection theorems to inductively prove Theorem \ref{thm:hyperplanesIncidence}, and we show how it can be applied to higher dimensional Furstenberg sets.

An $(s,t)$-hyperplane Furstenberg set in $\R^d$ is a set $F$ with the following property. There is a set $\H \subseteq \mathcal{A}(d,d-1)$ of affine hyperplanes in $\R^d$ with $\dim_H(\H) \geq t$ such that $\dim_H(F \cap H) \geq s$ for every $H\in\H$.

In a recent breakthrough, Ren and Wang \cite{RenWan23}, building on \cite{OrpShm23}, gave sharp bounds for $(s,t)$-Furstenberg sets in $\R^2$, proving the following theorem.
\begin{theorem}
    Let $s \in (0, 1]$, $t\in (0,2]$. If $F\subseteq \R^2$ is an $(s,t)$-Furstenberg set, then
    \begin{equation*}
        \dim_H(F) \geq \min\{s+t, \frac{3s+t}{2}, s+1\}.
    \end{equation*}
\end{theorem}

Furstenberg sets are less understood in higher dimensions, although there has been exciting progress \cite{BriDha25, BriOrtZak25, DabOrpVil22, Fiedler25, Hera19, HerKelMat19, Oberlin07, Ren23}. In Section \ref{sec:IncidenceHyperplanes}, we generalize Ren and Wang's result to hyperplanes in $\R^d$:
\begin{theorem}\label{thm:FurstenbergHyperplanes}
    Let $s \in (d-2, d-1]$ and  $t \in (0, d]$. If $F$ is an $(s,t)$-hyperplane Furstenberg set in $\R^d$, then
    \begin{equation*}
        \dim_H(F) \geq \min\{s+t, \frac{3s + t - (d-2)}{2}, s+1\}.
    \end{equation*}
\end{theorem}

\medskip
In Sections \ref{sec:LipschitzFunction} and \ref{sec:algorithmicDimDistances} we apply our incidence bound to Falconer's (pinned) distance problem in $\R^3$ and $\R^4$. If $E\subseteq \R^d$ and $x\in\R^d$, the pinned distance set of $E$ with respect to $x$ is
\begin{equation*}
    \Delta_x(E) := \{\vert x - y\vert\mid y\in E\}.
\end{equation*}
The Falconer pinned distance conjecture is a generalization of the important Falconer distance set conjecture \cite{Falconer85Dis}. It states that, for any Borel $E\subseteq\R^d$, if $\dim_H(E) > d/2$, then $\Delta_x(E)$ has positive Lebesgue measure for some $x\in E$. This remains an important open problem, for every $d \geq 2$. There has been substantial progress on this problem \cite{ChoCsoLutLutMayStu25, DuOuRenZha23Hig, DuGutOuWanWilZha21, FieStu23, FieStu24Pin, GutIosOuWan20, Harris21, Liu20}. 

In this paper, we study the dimensional version of Falconer's conjecture. Given any Borel set $E\subseteq\R^d$ with $\dim_H(E) > d/2$, our aim is to find a lower bound on the quantity
\begin{equation}\label{eq:target}
    \sup\limits_{x\in E} \dim_H(\Delta_x(E)).
\end{equation}
 For sets of equal Hausdorff and packing dimension, Shmerkin and Wang, generalizing the planar results in \cite{Orponen17Dis, Shmerkin17Dis, Shmerkin19Pin}, proved in \cite{ShmWan25Dis} that this quantity is maximal in every dimension, i.e., 
 $$\sup\limits_{x\in E} \dim_H(\Delta_x( E))=1,$$for every $d\geq 2$ and for every $E\subset\R^d$ with $\dim_H(E)=\dim_P(E) > d/2$.
However, for general sets this is still an open problem. In $\R^2$, the current best known bound is $3/4$, due to Stull \cite{Stull26}. In higher dimensions, the best preciously known bound was recently established by Du, Ou, Ren and Zhang \cite{DuOuRenZha23Hig}. They proved the following theorem.
\begin{theorem}
     Let $d\geq 2$, $0<\alpha \leq d-1$, and $E\subseteq\R^d$ be compact. If $\dim_H(E) =\alpha$, then 
     \begin{equation*}
         \sup\limits_{x\in E} \dim_H(\Delta_x( E))\geq \min\{\alpha\frac{2d+1}{d+1} - (d-1), 1\}.
     \end{equation*}
 \end{theorem}
 In particular, if $\alpha = d /2$, 
 \begin{equation*}
         \sup\limits_{x\in E} \dim_H(\Delta_x (E))\geq \frac{d+2}{2(d+1)}.
     \end{equation*}

In particular, the previous best bound for \eqref{eq:target}, for $E\subseteq \R^3$ with $\dim_H(E) > 3/2$ was $5/8$. The best known bound for \eqref{eq:target}, for $E\subseteq \R^4$ with $\dim_H(E) > 2$ was $3/5$. 

In this paper, we use our hyperplane incidence bound, combined with algorithmic techniques and Ren's radial projection theorem \cite{Ren23} to improve the best known bounds for $d = 3,4$. We prove that:
\begin{theorem}\label{thm:mainthm1}
    Let $E\subseteq\R^d$ be a Borel set such that $\dim_H(E) > d/2$, where $d=3$ or $d=4$. Then
    \begin{equation}
        \sup\limits_{x\in E} \dim_H(\Delta_x (E)) \geq 2/3.
    \end{equation}
\end{theorem}

\medskip
Our application of the hyperplane incidence bound to the Furstenberg set problem is straightforward. However, the application of the incidence bound to the distance set problem is much more involved. In addition to the incidence bound, a key ingredient of our distance set result is a partitioning argument, similar to that of used by the second author in \cite{Stull26}. Later a similar decomposition was also used, in a crucial way, in the solution of the Kakeya problem in $\R^3$ by Wang and Zahl \cite{WanZah25Vol}.

In Section \ref{sec:LipschitzFunction} we describe our decomposition result. We decompose the Kolmogorov complexity function $K_n(x)$ (as a function of $n$) into intervals which obey Frostman or Katz-Tao type non-concentration conditions. The main tool we use in finding our decomposition is Jones' traveling salesman theorem in \cite{jones1990rectifiable}. In our proof of the distance set estimate in Section \ref{sec:algorithmicDimDistances} we choose our partition in such a way that each interval corresponds to a term in the minimum of \eqref{eq:hyperplanesIncidenceBound}.

\medskip
Finally, we remark that the reason we limit ourselves to dimensions 3 and 4 is due to \eqref{eq:hyperplanesIncidenceBound}, where we need $d/2 \geq d-2$, which is only true when $d\leq 4$. However, our proof also gives some bound in higher dimensions, so long as we assume that the set $E$ has sufficiently high dimension.

\section{Preliminaries}

\subsection{Discretized incidences}\label{ssec:prelimDiscretized}
We will write $A \lesssim B$ to denote $A \leq c B$, for some absolute constant $c$. We will use the subscript notation $A \lesssim_n B$ to denote $A \leq C_n B$, where $C_n$ is a constant depending on $n$. In a similar way, we will use $\gtrsim$ and $\gtrsim_n$. We write $A \sim B$ if $A \lesssim B$ and $B \lesssim A$, and similarly for $\sim_n$. 

Let $\delta$ be a dyadic scale, i.e., $\delta=2^{-n}$ for some $n\in\N$. For any set $A\subset\R^d$, we denote by $\D_\delta(A)$ the set of half-open, half-closed dyadic cubes of side length $\delta$ that intersect $A$. We will also use the notation $\D_\delta:=\D_\delta(\R^d)$.

We call a set of hyperplanes a $\delta$-hyperplane if it is the dual set of the points covered by a cube $Q\in\D_\delta$. We denote the set of $\delta$-hyperplanes by $\H^\delta$. 
We say that a $\delta$-cube $Q$ intersects a $\delta$-hyperplane $H$, if $H$ contains a hyperplane that meets $Q$, and we denote this by $H\cap Q\neq \emptyset$. In the special case when $d = 2$, we call $\delta$-hyperplanes $\delta$-tubes. 

In \cite{RenWan23} Ren and Wang proved sharp bounds on the incidences between $\delta$-tubes and $\delta$-cubes in $\R^2$, leading to a resolution of the Furstenberg set problem in the plane. They introduced the following definition, and proved Theorem \ref{thm:RenWang} below (see Theorem 4.1 in \cite{RenWan23}).

Let $s\in[0,d]$ and $c > 0$. A non-empty bounded set $A\subseteq\R^d$ is called a $(\delta, s, c)$-set if it satisfies the following Frostman-type condition. For every integer $0\leq m\leq n$ and each $Q\in \D_{2^{-m}}$, $\D_\delta(A\cap Q)$ contains at most a $c2^{-sm}$ portion of all the cubes in $\D_\delta(A)$.

\begin{remark}
There is an alternative definition of $(\delta, s, c)$-sets in the literature. For example, in Ren and Wang's solution to the Furstenberg set problem \cite{RenWan23}, $A$ is called a $(\delta, s, c)$-set if for every ball $B$ of radius $r \in [\delta, 1]$, $\D_\delta(A\cap B)$ contains at most a $cr^s$ portion of all the cubes in $\D_\delta(A)$. It is immediate that the two definitions are equivalent in the sense that, if $A$ is a $(\delta, s, c)$-set under one definition, then it is a $(\delta, s, c^\prime)$-set in the other, where $c'$ and $c$ agree up to a multiplicative absolute constant that depends only on the ambient dimension $d$. 
\end{remark}

When $\P\subseteq \D_\delta$, we say that $\P$ is a $(\delta, s, c)$-set if the set of points in $\R^d$ covered by the cubes in $\P$ is a $(\delta, s, c)$-set. A set $\H$ of $\delta$-hyperplanes in $\R^d$ is called a $(\delta, s, c)$-set if its dual is a $(\delta, s, c)$-set.

\medskip

We will use the following theorem due to Ren and Wang \cite{RenWan23} which gives sharp bounds on the incidences between tubes and cubes in the plane. This will be used as our base case in the proof of Theorem \ref{thm:hyperplanesIncidence}.
\begin{theorem}\label{thm:RenWang}
   For every $\epsilon > 0$, $s\in (0,1]$, and $t\in (0, 2]$, there exists $\eta > 0$ such that for every small enough $\delta>0$ and for every $M\in\N$ the following holds.
   
   Suppose that  $\mathcal{P}\subseteq\D_\delta([0,1]^2)$ is a $(\delta, t, \delta^{-\eta})$-set, and suppose that for each $Q\in\P$ there is a $(\delta, s, \delta^{-\eta})$-set of tubes $T(Q)$ such that each tube in $T(Q)$ intersects $Q$, and $|T(Q)|\sim M$. Then for $\T:=\bigcup_{Q\in\P} T(Q)$ we have
   \begin{equation}
       \vert \mathcal{T}|\gtrsim_\eps \delta^{-\min\{t, \frac{s+t}{2}, 1\} + \epsilon} M.
   \end{equation}
\end{theorem}

\bigskip

\subsection{Basic algorithmic methods}\label{ssec:prelimAlgMethods1}
Our proof will use the algorithmic methods which have recently been applied to geometric measure theory \cite{AltBusWil25,BusFie25,Lutz21, LutQiYu24, LutLut18, LutLutMay23Ext}. In particular, we rely on the techniques introduced by the authors in \cite{CsoStu25}. We refer the reader to Section 2.2. of \cite{CsoStu25} for a summary of the fundamental definitions, notations and results central to this research area. In this section, we present only the specific techniques and prior results necessary for our arguments that were not already introduced in Section 2.2 of \cite{CsoStu25}. 

Let $x \in [0,1]^{d}$. It will be important for us to understand the behavior of the complexity function $K_n(x)$ as a function of $n$. Note that this function is monotone and $d$-Lipschitz, up to a logarithmic error. It is an easy exercise that this implies that we can approximate any complexity function by a monotone, $(d+1)$-Lipschitz, piecewise linear function $f: [0,\infty) \to [0, \infty)$ such that
\begin{equation}
    \left| f(n) - K_n(x) \right| \lesssim \log n,
\end{equation}
for every $n\in\N$.

In the proof of our distance set theorem, we will discuss the complexity of unit vectors in $\R^d$. We will use the standard fact that, for every dyadic scale $\delta$, we can cover $\S^{d-1}$ by $\delta$-caps with bounded overlap, which we denote $\D_\delta(\S^{d-1})$. For every scale, we can fix a representative element $e$ for every $\delta$-cap. Moreover, each representative $e$ is computable. Therefore, if $e$ is any representative vector at scale $2^{-m}$, then, for every $n\geq m$, $K_{n,m}(e\mid e) \approx 0$. Additionally, if $e^\prime$ is any direction, and $m$ is any precision, given $e^\prime$ with precision $m$, we can computably assign to $e^\prime$ a representative $e$ which agrees with $e^\prime$ up to precision $m$.

Recall that the \textit{symmetry of information}, proved in \cite{LutStu20}, says that
\begin{equation*}
    K^A_{n,m}(x, y) \approx K^A_{n, m}(x\vert y) + K^A_m(y)
\end{equation*}
for every oracle $A$, $x\in\R^{d_1}, y\in\R^{d_2}$ and precisions $n,m \in \N$. We will use several simple consequences of the symmetry of information throughout the paper. The first is that a point $x$ doesn't give any information about a point $y$ if and only if $y$ doesn't give information about $x$. In particular, we have the following. Let $A$ be an oracle, $x, y \in \R^d$, $\eps > 0$. Then $K^A_{m,n}(y \mid x) \geq K^A_m(y) - \eps n$ iff $K^A_{n,m}(x \mid y) \gtrapprox K^A_n(x) - \eps n$. Additionally, $K^A_{k,n}(y \mid x) \gtrapprox K^A_k(y) - \eps n$ for every $k\le m$. 

We will also routinely use the fact that, if a set $S$ is enumerable then all elements of $S$ have complexity, up to logarithmic error, at most $\log |S|$. By an enumerable set, we mean that there is an algorithm that can list the elements of this set. This algorithm does not need to stop after finitely many steps and know that it already finished writing $S$.  Moreover, $\gtrsim 2^{-m}|S|$ many elements of $S$ have complexity greater than $\log |S| - m$ for every $m$. A very useful consequence of this fact is the following. Let $x$ be a point of high complexity, and suppose $x$ is contained in an enumerable set. Then, if $x$ has a computable property, then many other points of that set must also have that property. 

\subsection{Algorithmic incidence results}

In this section, we present some tools that were developed in \cite{CsoStu25, FieJin26, LutStu20}. We state and prove versions of these results needed in this paper. 

We will need the following proposition that relates the complexity of a point with $(\delta, s, c)$-sets. Its proof is largely identical to the proof of Proposition 3.4 in \cite{CsoStu25}, we include it for completeness.
\begin{proposition}\label{prop:enumerableSset}
    Let $A$ be an oracle, $0<t\leq d$ and $0 < \eta \leq 1$. Then for every sufficiently large $n$, depending on $t$ and $\eta$, the following holds.
    
    Let $x\in\R^d$, and suppose that
    \begin{equation*}
        K^A_m(x) \geq tm - \eta n
    \end{equation*}
    for every $0 \leq m \leq n$.
    
    If $\P\subseteq \D_{2^{-n}}(\R^d)$ is enumerable given $w$, relative to $A$, and $\P$ contains the $2^{-n}$-cube containing $x$, then it has a  $(2^{-n}, t, c)$ subset $\P'$ such that $c= 2^{O(\eta n)+|w|}$ and
    \begin{equation*}
        \vert \P^\prime\vert \geq 2^{K^A_n(x)-O( \eta^{-1}\log n) -|w|}.
    \end{equation*}
    In particular, if $|w| = O(\log n)$, then $\vert \P^\prime\vert \ge 2^{K^A_n(x) - O(\eta^{-1} \log n)}$ and $c = 2^{O(\eta n)}$.
\end{proposition}
\begin{proof} 
We define $\P'$ to be the set of those cubes $Q$ in $\P$ for which $$K^A(Q)\leq K^A_n(x),\,\text{ and }\,K^A_{n,m}(Q\mid Q)\leq K^A_{n,m}(x\mid x)$$ for every $m\le n$ which is an integer multiple of $\lfloor \eta n\rfloor$. 
 
    Note that $\P^\prime$ is enumerable given $w$, given the numbers $K^A_{n,m}(x\mid x)$ for each integer multiple $m$ and given $K^A_{n}(x)$. Moreover, $Q(x) \in \P^\prime$, where $Q(x)$ is the $2^{-n}$-dyadic cube containing $x$. Therefore, 
    \begin{align*}
        K^A_n(x) &\lessapprox \log \vert \P^\prime \vert + \log K^A_n(x) + \sum\limits_{m} \log K^A_{n,m}(x\mid x)+|w|\\
        &\leq \log\vert \P^\prime \vert  + O(\eta^{-1}\log n)+|w|.
    \end{align*}
 
    Hence, 
    \begin{equation*}
        \vert \P^\prime\vert \geq 2^{K^A_n(x)-O(\eta^{-1}\log n)-|w|}
    \end{equation*}
    as required. Moreover, since there are at most $2^{K^A_n(x) + O(\log n)}$ many cubes $Q$ such that $K^A(Q)\leq K^A_n(x)$, we see that $$|\P'|\le 2^{K^A_n(x) + O(\log n)}.$$

    Let $m \leq n$ be an integer multiple of $\lfloor \eta n\rfloor$. Now let $Q\in \D_{2^{-m}}$. By our definition of $\P^\prime$, if $Q'\in\P'$ is a subcube of $Q$, then $K^A_{n,m}(Q'\mid Q)=K^A_{n,m}(Q'\mid Q')\leq K^A_{n,m}(x\mid x)$. For any fixed $Q\in \D_{2^{-m}}$, the number of subcubes $Q^\prime \in \D_{2^{-n}}(Q)$ such that $K^A_{n,m}(Q^\prime \mid Q) \lessapprox K^A_{n,m}(x\mid x)$ is $\lesssim 2^{K^A_{n,m}(x\mid x) + O(\log n)}$. Therefore, using symmetry of information, the number of subcubes is at most
     $$2^{K^A_{n, m}(x\mid x)+O(\log n)}\leq 2^{K^A_n(x) - K^A_{m}(x)+O(\log n)}\le \vert \P^\prime\vert 2^{-tm+\eta n +O(\eta^{-1}\log n)+|w|}.$$

Since this holds for every $m$ which is an integer multiple of $\lfloor \eta n\rfloor$, and since between $m$ and $m+\lfloor \eta n\rfloor$ the number of cubes can be multiplied by at most a factor $2^{O(\eta n)}$, the statement indeed follows with $c=2^{O(\eta n)+O(\eta^{-1}\log n)+|w|}=2^{O(\eta n)+|w|}$.
\end{proof}

The following technical proposition allows us to find `random' points of $(\delta, s, c)$-sets with nice properties in terms of the algorithmic methods.
\begin{proposition}\label{prop:sSetToComplexities}
    Let $\delta=2^{-n}$ be a dyadic scale, $c > 0$, $s > 0$. Suppose that $\P\subseteq \D_\delta([0,1]^d)$ is a $(\delta, s, c)$-set. Let $A$ be any oracle such that $c, s$ and $\P$ are computable. Then for every oracle $B$ there is a subset $\P^\prime \subseteq \P$ such that $|\P^\prime| \gtrsim |\P|$ and for every $Q\in \P^\prime$ the following hold.
    \begin{enumerate}
        \item[(i)] $K^{A}_n(Q) \approx \log \vert \P \vert$.
        \item[(ii)] $K^A_m(Q) \gtrapprox sm - O(\log c)$ for all $m \leq n$.
        \item[(iii)] $K^{A,B}_m(Q) \approx K^A_m(Q)$ for all $m\leq n$.
    \end{enumerate}
\end{proposition}
\begin{proof}
    There are at most $2^{\log |\P| - \log n}$ many dyadic cubes of complexity less than $\log |\P| - \log n$, that is, the size of the set of those cubes $Q$ for which $K^A_n(Q) < \log |\P| - \log n$ is at most $|\P|/n$. And since $\P$ is computable, relative to $A$, therefore $K^A_n(Q) \lessapprox \log |\P|$ for every $Q\in\P$. So (i) holds for every $Q\in\P$ outside of a set of size at most $|\P|/n$.
    
    For every $m \leq n$, let 
    \begin{equation*}
        N_m = \{Q\in \D_{2^{-m}} \mid K^A_m(Q) < sm - \log c - 2\log m -3\}.
    \end{equation*}
    Then the cardinality of $N_m$ is at most $2^{sm - \log c - 2\log m - 3}$. Since $\P$ is a $(\delta, s, c)$-set, the set 
    \begin{equation*}
        \{Q \in \P\mid K^A_m(Q)< sm - \log c - 2\log m-3\} 
    \end{equation*}
    has cardinality at most 
    \begin{equation*}
        |N_m|\, c 2^{-sm}|\P| \leq |\P|/(8m^2).
    \end{equation*}
    Summing over $m$, we conclude that (ii) holds for at least $3|\P|/4$ many cubes in $\P$. 
    
    To prove the final conclusion, let $B$ be any oracle. It was shown in \cite{Stull26Opt}, Lemma 12, that there are at most $|\P|/4$ many cubes $Q\in \P$ for which there is an $m\le n$ with $K^{A,B}_m(Q) < K^A_m(Q) - 2\log m - 2$.
\end{proof}

The following projection lemma was essentially proved by Fiedler and Jing (Theorem 24 of \cite{FieJin26}), and will be used in proving the discretized slicing and projection results in Section \ref{sec:discretizedSlicingProjection}. We denote the set of unit vectors in $\R^d$ by $\S^{d-1}$. For any $V \in G(d, k)$, we denote the orthogonal projection of $\R^d$ onto $V$ by $\pi_V$. When projecting onto one-dimensional lines, we also denote the orthogonal projection map by $p_e$, where $e$ is the direction of the line.

\begin{lemma}\label{lem:algSlicing}
    Let $s > 0$, $\eps > 0$ such that $s,\eps\in\Q$. Then, for every sufficiently large $n$ the following holds. If $A$ is an oracle, $x \in \R^d$ and $e\in \mathcal{S}^{d-1}$ satisfy
    \begin{enumerate}
        \item[(i)] $K^A_m(e) \gtrapprox (d-1)m -\eps n$, for all $m \leq n$,
        \item[(ii)] $K^{A,e}_m(x) \gtrapprox K^A_m(x) - \eps n$, for all $m\leq n$,
        \item[(iii)] $K^A_m(x) \gtrapprox sm - \eps n$, for all $m\leq n$,
    \end{enumerate}
    then, for every $m\leq n$,
    \begin{equation*}
        K^{A,e}_{m}(\pi_{e^\perp}\, x) \gtrapprox \min\{sm, (d-1)m\} - O(\eps n).
    \end{equation*}
    
    Moreover, if $s > 1$, then, for every $m\leq n$,
    \begin{equation*}
        K^{A,e}_{n,n,m}(x\mid p_e x, x) \lessapprox K^A_{n}(x) - n -(s-1)m + O(\eps n).
    \end{equation*}
    
\end{lemma}

For completeness, we will give a proof of Lemma \ref{lem:algSlicing}, which largely follows Lemma 4.2 of \cite{LutStu24} and Theorem 24 of \cite{FieJin26}. We will need the following auxiliary lemma, which shows that we can use an oracle to lower the complexity of points at specified precisions using only the information already present in the point, i.e., without affecting the complexity of other objects. 
\begin{lemma}[\cite{LutStu24}, Lemma 3.4]\label{lem:oracles}
Let $A$ be an oracle, $x\in\R^d$, $\eta>0$, $\eta\in\Q$, and $n\in\N$. Suppose that $\eta n \leq K^A_n(x)$. Then there is an oracle $D = D(n, \eta)$ with the following properties.
\begin{itemize}
\item[\textup{(i)}] For every $m\leq n$,
\[K^{A,D}_m(x)=\min\{\eta n,K^A_m(x)\}+O(\log n)\,.\]
\item[\textup{(ii)}] For every $d^\prime,m\in\N$ and $y\in\R^{d^\prime}$,
\[K^{A,D}_{m,n}(y\mid x)=K^A_{m,n}(y\mid x)+ O(\log n)\,,\]
and
\[K_m^{A,x,D}(y)=K_m^{A,x}(y)+ O(\log n)\,.\]
\item[\textup{(iii)}] If an oracle $B$ satisfies $K^{A,B}_n(x) \geq K^A_n(x) - O(\log n)$, then \[K_n^{A,B,D}(x)\geq K_n^{A,D}(x) - O(\log n)\,.\]
\item[\textup{(iv)}] For every $m\in\N$, $u\in\R^d, w\in\R^{d^\prime}$
\[K^A_{n,m}(u\mid w) \leq K^{A,D}_{n,m}(u\mid w) + K^A_n(x) - \eta n + O(\log n)\,.\]
\end{itemize}

\end{lemma}

With this technical lemma, we are able to prove Lemma \ref{lem:algSlicing}.
\begin{proof}[Proof of Lemma \ref{lem:algSlicing}]
    We first prove the moreover claim. Assume that (i), (ii) and (iii) are satisfied, and $s > 1$. Let $\eta\in\Q$ such that $$\eta n = n + (s-1)m - 6\eps n.$$ Let $D$ be the oracle given by Lemma \ref{lem:oracles}. Note that $\eta n \leq K^A_n(x)$ (for $n$ sufficiently large), since, by assumption (iii), $K^A_n(x) \gtrapprox sn - \eps n$, and since $s > 1$, therefore $\eta n\le sn-6\eps n$.
    
    Also note that, by (ii), $K_{n,m}^A(x\mid e)\gtrapprox K^A_n(x)-\eps n$, and using the symmetry of information and (i) this implies that
    \begin{equation}\label{eq:xGivesNoInfoAboute}
        K^A_{r,n}(e\mid x)\gtrapprox K^A_{r}(e) -\eps n\gtrapprox (d-1)r -2\eps n,
    \end{equation}
    for every $r\leq n$.
    
    Let $Q_m$ be the dyadic cube of side length $2^{-m}$ containing $x$. We now show that $x$ is essentially the unique point in $Q_m$ whose projection onto $e$ is $p_e x$. More precisely, we will show that if $y\in Q_m$, $p_e y = p_e x$, and $|x-y|>2^{-n+3\eps n}$, then $K^{A,D}_n(y)> K^{A,D}_n(x)$. 
    
    Suppose that $K^{A,D}_n(y)\le K^{A,D}_n(x)$. It is not difficult to verify that, given $x$ and $y$ with precision $n$, we can compute the direction between $x$ and $y$ with an error at most $2^{-(n-t)}$, i.e. with precision $n-t$, where $t = -\lfloor\log |x-y|\rfloor$ (note that, by our assumption, $n-t\ge 0$.)
    
    Since this direction belongs to $e^{\perp}$, and the space of all hyperplanes that contain a given direction is a $d-2$-dimensional space, therefore, to find $e$ with 
    precision $n-t$, it is sufficient to add $(d-2)(n-t)$ extra information.
    
    In the language of Kolmogorov complexity, this shows that
    \begin{equation}\label{eq:computeEgivenXY}
        K^{A,D}_{n-t, n}(e\mid x) \lessapprox K^{A,D}_n(y\mid x) + (d-2)(n-t).
    \end{equation}
    We note that $K^{A,D}_t(x\mid y) \approx K^{A,D}_t(y\mid x) \approx 0$, since, at scale $2^{-t}$, $x$ and $y$ are contained in $2^{-t}$-dyadic cubes which are close together. Using this fact, and the symmetry of information, we conclude that
    \begin{align*}
        K^{A,D}_{n-t, n}(e\mid x) &\lessapprox K^{A,D}_{n,t}(y\mid x) + (d-2)(n-t)\\
        &\approx K^{A,D}_n(y) - K^{A,D}_t(x)+(d-2)(n-t),
    \end{align*}
    where $$K^{A,D}_n(y) \le K^{A,D}_n(x)\le\eta n=n + (s-1)m - 6\eps n.$$
    Applying Lemma \ref{lem:oracles}(ii) and \eqref{eq:xGivesNoInfoAboute}, we see that
    $$K^{A,D}_{n-t, n}(e\mid x)\gtrapprox K^{A}_{n-t, n}(e\mid x)\gtrapprox(d-1)(n-t)-2\eps n.$$
Putting together these inequalities we get
\begin{align*}        K^{A,D}_t(x) &\le n + (s-1)m - 6\eps n+(d-2)(n-t)-(d-1)(n-t)+2\eps n+O(\log n)\\
&=(s-1)m+t-4\eps n+O(\log n).\label{m}\\
\end{align*}

Recall that $x$ and $y$ belong to the same dyadic cube of size $2^{-m}$, therefore $t \geq m -c$, for some absolute constant $c$. Also recall that by Lemma \ref{lem:oracles}(i), $$K^{A,D}_t(x) = \min\{\eta n, K^A_t(x)\} + O(\log n).$$ If the minimum is $K^A_t(x)$, then we have an immediate contradiction, since $K^A_t(x) \gtrapprox st - \eps n$ by condition (iii), and therefore we would have 
$$st-\eps n\le (s-1)m + t  - 4\eps n +O(\log n)$$ contradicting $t\ge m-c$ (for every large enough $n$).

 If $K^{A,D}_t(x) = \eta n+O(\log n)$, then from our upper estimate of $K^{A,D}_t(x)$ above we get $ t \ge n-2\eps n-O(\log n)$. Therefore $y$ is within $2^{-n+2\eps n+O(\log n)}$ of $x$, contradicting our initial assumption that $|x-y|\ge 2^{-n+3\eps n}$.
 
 Using the fact that $x$ with precision $n$ is, essentially, uniquely described by $Q\in\D_{2^{-m}}$, its projection to $e$, and by $K_n^{A,D}(x)$ (more precisely, by our argument above, there are at most $2^{O(n\eps)}$ many other $y$ with the same properties), a standard enumeration argument (see, e.g., Lemma 3.1 of \cite{LutStu24}) shows that we can find $x$ with precision $n$ with at most an additional $O(\eps n)$ information. That is,

    \begin{equation}
        K^{A,D,e}_{n,n,m}(x \mid p_e x,  x) \lessapprox O(\eps n).
    \end{equation}
    The conclusion follows by Lemma \ref{lem:oracles}(iv), since
    \begin{align*}
        K^{A,e}_{n,n,m}(x \mid p_e x, x)  &\lessapprox K^{A,D,e}_{n,n,m}(x \mid p_e x, x) + K^A_n(x) - \eta n\\
        &\lessapprox K^A_n(x) - n - (s-1)m + O(\eps n), 
    \end{align*}
    as required.

    We now turn to the first claim of Lemma \ref{lem:algSlicing}. The proof is largely identical to the argument above. Let $m \leq n$. We note that we may assume that $sm \geq O(\eps n)$, otherwise the conclusion is trivially satisfied. Let $\eta\in\Q$ such that $$\eta m = \min\{sm, (d-1) m\} - 6\eps n >0.$$ Let $D$ be the oracle given by Lemma \ref{lem:oracles} with, importantly, $n$ replaced by $m$ in the assumptions of the lemma (and $m$ replaced by a smaller integer in its conclusions). Note that we can indeed apply Lemma \ref{lem:oracles}, since by (iii) of Lemma \ref{lem:algSlicing}, $$K_m^A(x)\geq sm - \eps n - O(\log n) \geq \eta m.$$

    We again note that, by (i), (ii) and the symmetry of information (by the same argument as we obtained \eqref{eq:xGivesNoInfoAboute} before), we have
    \begin{equation}\label{eq:modified5}
        K^A_{r,m}(e \mid x) \gtrapprox (d-1)r -2\eps n,
    \end{equation}
    for all $r \leq m \leq n$. Our goal is to show that $x$ is essentially the unique low complexity point, at precision $m$, whose projection onto $e^\perp$ is $\pi_{e^\perp}\, x$.  Let $y\in \R^d$ such that $\pi_{e^\perp} \,y = \pi_{e^\perp}\, x$ and $K^{A,D}_m(y) \leq K^{A,D}_m(x)$. As before, since given $x$ and $y$ with precision $m$, we can compute the direction between $x$ and $y$ with an error a $2^{-(m-t)}$ (provided that $t: = \lfloor-\log |x-y|\rfloor\le m$), therefore
    \begin{equation}\label{eq:computeEgivenXY2}
        K^{A,D}_{m-t, m}(e\mid x) \lessapprox K^{A,D}_m(y\mid x).
    \end{equation}
    Via the symmetry of information, and using the fact that $K_{t}(x\mid y) \approx K_{t}(y\mid x) \approx 0$, we conclude that
    \begin{align*}
        K^{A,D}_{m-t, m}(e\mid x) &\lessapprox K^{A,D}_{m,t}(y\mid x)
        \approx K^{A,D}_m(y) - K^{A,D}_{t}(x)
        \lessapprox \eta m - K^{A,D}_{t}(x).
    \end{align*}
    By \eqref{eq:modified5}, Lemma \ref{lem:oracles}(ii), and rearranging the above inequality, we see that
    \begin{align*}
        K^{A,D}_{t}(x) &\lessapprox \eta m - K^{A,D}_{m-t, m}(e\mid x)\lessapprox \eta m - K^{A}_{m-t, m}(e\mid x)\\
        &\le \eta m- (d-1)(m-t)  + 2\eps  n.
    \end{align*}
    
    By Lemma \ref{lem:oracles}(i), $$K^{A,D}_{t}(x) = \min\{\eta m, K^A_{t}(x)\} + O(\log m).$$If the minimum is $K^A_{t}(x)$, then
    \begin{equation}
        st - \eps n \le\eta m - (d-1)(m-t)  + 2\eps n +O(\log m),
    \end{equation}
    or, equivalently, 
    \begin{equation}
        (d-1 - \eta) m \le (d-1 - s)t + 3\eps n +O(\log m).
    \end{equation}
    If $s \leq d-1$,  i.e. $\eta m= sm-6\eps n$, then this implies that 
    \begin{equation}
        (d-1 - s)m +3\eps n \le (d-1 - s)t + 2\eps n +O(\log m), 
    \end{equation}
    a contradiction. If $s > d-1$, i.e. $\eta m= (d-1)m-6\eps n$, then we would have 
    \begin{equation}
        3\eps n \le (d-1-s)t +2\eps n+O(\log m)\le 2\eps n+O(\log m),
    \end{equation}
    again a contradiction.
    
    If $K^{A,D}_{t}(x) = \eta m +O(\log m),$ then  we see that
    \begin{equation}
        \eta m\le \eta m-(d-1)(m-t)+2\eps n +O(\log m),
    \end{equation}
     and since $d-1\ge 1$, this implies that $t \ge m- O(\eps n)$. Therefore $y$ is within $2^{-m + O(\eps n)}$ of $x$. Again, a simple enumeration argument shows that this implies that
    \begin{equation}\label{mm}
        K^{A,D,e}_{m}(x \mid \pi_{e^\perp}\,x) \lessapprox O(\eps n).
    \end{equation}
   The conclusion follows by symmetry of information, Lemma \ref{lem:oracles}(iv), \eqref{mm}, and Lemma \ref{lem:algSlicing}(ii) as

\begin{align*}
        K^{A,e}_m(\pi_{e^\perp}\, x)&\approx K^{A,e}_m(x)-K^{A,e}_{m}(x \mid \pi_{e^\perp}\, x)\\
        &\gtrapprox K^{A,e}_m(x)-K^{A,D,e}_{m}(x \mid \pi_{e^\perp}\, x) - K^A_m(x) + \eta m\\
        &\ge  K^{A,e}_m(x)-O(\eps n)-K^A_m(x)+ \eta m\\
        &\gtrapprox -O(\eps n)+ \eta m,
    \end{align*}

    as required. 
\end{proof}

We end this section with a slight generalization of Lemma \ref{lem:algSlicing}. Its proof is nearly identical to that of Lemma \ref{lem:algSlicing}.
\begin{lemma}\label{lem:algProjection}
    Let $s > 0$, $A$ be an oracle, $x \in \R^d$ and $e\in \mathcal{S}^{d-1}$ satisfying
    \begin{enumerate}
        \item[(i)] $\dim_H^{A,x}(e) > s$ and
        \item[(ii)] $\dim_H^A(x) > s$.
    \end{enumerate}
    Then, for all $\eps >0$ and all sufficiently large $n$, 
    \begin{equation*}
        K^{A}_{n}(x\mid \pi_{e^{\perp}}\, x, e) \lessapprox K^A_n(x) - sn + \eps n.
    \end{equation*}
\end{lemma}
\begin{proof}
    Let $\eta\in\Q$ such that $$\eta n = sn - 3\eps n.$$ Let $D$ be the oracle given by Lemma \ref{lem:oracles}. 
    
    As in the proof of Lemma \ref{lem:algSlicing}, we aim to show that $x$ is essentially the unique point in the ball of radius $1$ around $x$ whose projection onto $e^\perp$ is $\pi_{e^\perp} x$. 
    
    Suppose that $K^{A,D}_n(y)\le K^{A,D}_n(x)$ and that $\pi_{e^\perp} \,x$ and $\pi_{e^\perp} y$ agree at precision $n$. We again use the simple fact that, given $x$ and $y$ with precision $n$, we can compute the direction between $x$ and $y$ with an error at most $2^{-(n-t)}$, where $t = -\lfloor\log |x-y|\rfloor$. Note that this direction, up to precision $n-t$, is $e$.
    
    In the language of Kolmogorov complexity, this shows that
    \begin{equation}\label{eq:computeEgivenXY3}
        K^{A,D}_{n-t, n}(e\mid x) \lessapprox K^{A,D}_n(y\mid x).
    \end{equation}
    As before, $K^{A,D}_t(x\mid y) \approx K^{A,D}_t(y\mid x) \approx 0$. Using this fact, and the symmetry of information, we conclude that
    \begin{align*}
        K^{A,D}_{n-t, n}(e\mid x) &\lessapprox K^{A,D}_{n,t}(y\mid x)
        \approx K^{A,D}_n(y) - K^{A,D}_t(x),
    \end{align*}
    where $$K^{A,D}_n(y) \le K^{A,D}_n(x)\approx\eta n= sn - 3\eps n.$$
    Applying assumption (i) and Lemma \ref{lem:oracles}(ii), we see that
    $$K^{A,D}_{n-t, n}(e\mid x)\gtrapprox K^{A}_{n-t, n}(e\mid x)\gtrapprox s(n-t)-\eps n.$$
Putting together these inequalities we get
\begin{align*}        
K^{A,D}_t(x) &\le \eta n - s(n-t) + \eps n + O(\log n).\\
&\le st - 2\eps n + O(\log n).
\end{align*}

Recall that by Lemma \ref{lem:oracles}(i), $$K^{A,D}_t(x) = \min\{\eta n, K^A_t(x)\} + O(\log n).$$ If the minimum is $K^A_t(x)$, then we have an immediate contradiction by condition (ii). 

 If $K^{A,D}_t(x) = \eta n$, then from our upper bound of $K^{A,D}_t(x)$ above, we see that
 \begin{equation*}
     s(n-t) \leq \eps n + O(\log n)
 \end{equation*}
Therefore $y$ is within $2^{-n+(\eps/s) n+O(\log n)}$ of $x$.

 We again use the standard enumeration argument to conclude that we can find $x$ with precision $n$ with at most an additional $O(\eps n)$ information. That is,

    \begin{equation}
        K^{A,D}_{n}(x \mid \pi_{e^\perp}x,  e) \lessapprox O(\eps n).
    \end{equation}
    The conclusion follows by Lemma \ref{lem:oracles}(iv), since
    \begin{align*}
        K^{A}_{n}(x \mid \pi_{e^\perp}x,  e)  &\lessapprox K^{A,D}_{n}(x \mid \pi_{e^\perp}x,  e) + K^A_n(x) - \eta n\\
        &\lessapprox K^A_n(x) - sn + O(\eps n), 
    \end{align*}
    as required.
\end{proof}

\subsection{Deducing the distance set result}
In this section, we show that, in order to prove Theorem \ref{thm:mainthm1}, it suffices to prove its algorithmic analogue Theorem \ref{thm:maintheoremEff} below. 
This is a standard reduction, e.g. we used a very similar argument in \cite{CsoStu25}. For the sake of completeness, in this section we briefly present the variant that we use here.

\begin{theorem}\label{thm:maintheoremEff}
    Let $d = 3,4$,  $A$ be an oracle, $x, y\in\R^d$ and $e = \frac{y-x}{|x-y|}$. Suppose that 
    \begin{itemize}
\item[(i)] $\dim_H^A(x), \dim^A_H(y), \dim^{A,x}_H(e) > d/2$; 
\item[(ii)] $K^{A, x}_n(y) \gtrapprox K^{A}_n(y)$ for every $n\in\N$.
\end{itemize}
Then,
\begin{equation}
    \dim^{A,x}_H(|x-y|) \geq \frac{2}{3}.
\end{equation}
\end{theorem}

To prove the existence of such points, we will use Ren's radial projection result \cite{Ren23} combined with standard algorithmic tools. Recall that the Riesz $s$-energy of a finite, compactly supported Borel measure $\mu$ on $\R^d$ is
\begin{equation*}
    \mathcal{E}_s(\mu) = \iint \frac{d\mu(x) d\mu(y)}{|x-y|^s}.
\end{equation*}

\begin{proposition}[Proposition 9.1 in \cite{Ren23}]\label{prop:Ren}
Let $1\leq k \leq d-1$ be an integer, and fix $s\in (k-1, k]$ and $\eta > 0$. Let $\mu, \nu$ be Borel probability measures on $\R^d$ such that $\mathcal{E}_s(\mu),\mathcal{E}_s(\nu) < \infty$ and $\mu$ and $\nu$ have $\sim 1$ separated supports. Suppose that $\mu(H) = \nu(H) = 0$ for every $k$-plane $H$. Then, for $\mu$-almost every $x$, and for all Borel sets $Y$ of positive $\nu$ measure, 
\begin{equation*}
    \dim_H(\pi_x Y) \geq s - \eta,
\end{equation*}
where $\pi_x(y) = \frac{y-x}{|y-x|}$ is the radial projection map.
\end{proposition}

As we have already discussed it in Section 2.2 in \cite{CsoStu25}, if $Y\subseteq\R^d$ is computably compact relative to $A$ and $x\in\R^d$, then 
 \begin{equation}\label{don}
     \dim_H(\pi_x Y) = \sup\limits_{y\in Y} \dim^{A,x}_H(\pi_x y).
 \end{equation}

We also discussed in Section 2.2 of \cite{CsoStu25}, that, if $\mu$ is a probability measure whose support is contained in a compact set $X\subset\R^d$, then $X$ has a Hausdorff oracle $A$ relative to which $\mu$ is computable and such that, for every oracle $B$, 
\begin{equation}\label{dd}
    K^{A,B}_n(x) \gtrapprox K^A_n(x)
    \end{equation}
for $\mu$-almost every $x\in X$. In the following corollary, we choose $A$ such that this holds for $X$ with $\mu$, and it also holds for $Y$ with $\nu$.

\begin{corollary}\label{cor:existenceOfPoints}
    Let $X,Y\subseteq\R^d$ be compact sets and let $\mu, \nu$ be Borel probability measures whose support is contained in $X$ and $Y$, respectively such that $\mathcal{E}_{s}(\mu),\mathcal{E}_{s}(\nu) < \infty$ for some $s > d/2$ and $\mu$ and $\nu$ have $\sim 1$ separated supports. 
    Suppose that $\mu(H) = \nu(H) = 0$ for every affine hyperplane $H$. Then, there exist $x,y\in \R^d$ satisfying assumptions (i) and (ii) of Theorem \ref{thm:maintheoremEff}.
\end{corollary}
\begin{proof}
   Let $x\in X$ be a point such that $\dim_H^A(x) > d/2$ and such that $\dim_H(\pi_x Y^\prime) > d/2$ for every $Y^\prime$ of positive $\nu$ measure. Note that, by Proposition \ref{prop:Ren}, such a point exists. 

   Let $Y^\prime$ be the set of $y\in Y$ such that (ii) holds. By \eqref{dd}, $Y^\prime$ has full $\nu$-measure. Hence there must be a point $y\in Y^\prime$ such that $$\dim^{A,x}_H(\pi_x y) = \dim^{A,x}_H(e) > d/2,$$ and the proof is complete.
\end{proof}

We are now able to prove our main theorem on pinned distance sets in $\R^3$ and $\R^4$.
\begin{proof}[Proof of Theorem \ref{thm:mainthm1}]
    We proceed by induction on $d$, up to $d = 4$. When $d = 2$, the conclusion is well-known (see, e.g., \cite{KelShm19}).
    
    For the inductive step, suppose that $d > 2$. Since $E$ is Borel and $\dim_H(E) > d/2$, there is a $s > d/2$, compact subsets $X, Y\subseteq E$ and Borel probability measures $\mu$ and $\nu$ such that $\mathcal{E}_{s}(\mu), \mathcal{E}_{s}(\nu) < \infty$ and $\mu$ and $\nu$ have separated supports. 

    If either $\mu(H)$ or $\nu(H)$ is non-zero for some hyperplane $H$, then by our inductive hypothesis, the proof is finished. We can therefore assume that $\mu(H) = \nu(H) = 0$ for every hyperplane $H$, and so we can apply Corollary \ref{cor:existenceOfPoints}. Hence, there are points $x\in X$ and $y\in Y$ satisfying (i) and (ii). Applying Theorem \ref{thm:maintheoremEff}, we see that $\dim_H^{A,x}(|x-y|) \geq 2/3$. Since $Y$ is computably compact relative to $A$, and $y \mapsto |x-y|$ is computable relative to $x$, $(A,x)$ is a Hausdorff oracle for $\Delta_x Y$, and the conclusion follows.
\end{proof}

\section{Discretized slicings and projections}\label{sec:discretizedSlicingProjection}

In this section we prove a few facts about discretized slicings and projections. 

\medskip

Let $\delta=2^{-n}$ be a dyadic scale. For any hyperplane $H$ we denote its normal vector by $H^\perp$. 

If $H$ is a $\delta$-hyperplane, the normal vectors of the hyperplanes belonging to $H$ agree up to an error at most $\lesssim \delta$. We denote by $H^\perp$ the normal vector of $H$, which represents a direction defined up to an error at most $\lesssim \delta$. For any linear subspace $V$, we write $H^\perp\subset  V$ if $H^\perp$ agrees with a direction in $V$ up to an error at most $\lesssim\delta$. More generally, for another $\delta$-hyperplane $H'$, the notation $H^\perp\subset H'$ indicates that $H'$ contains both the origin and a direction agreeing with $H^\perp$, both up to an error at most $\lesssim\delta$. 

\medskip
Let $\H$ be an arbitrary set of $\delta$-hyperplanes in $\R^d$. Let $A$ be an oracle relative to which $\H$ is computable. We fix an $\eps>0$, and a direction
$e\in\S^{d-1}$. 

Let $\H'$ denote the set of those $\delta$-hyperplanes $H\in\H$ for which $H^\perp\subset e^\perp$. Then either $\H'$ is empty, or we can find an $H\in \H'$ with $K_n^{A,e}(H)\gtrapprox \log |\H'|$. Trivially, we also have $K_n^A(H)\lessapprox\log|\H|$. Using symmetry of information, we get
\begin{align*}
\log|\H'|&\lessapprox K_n^{A,e}(H)\lessapprox K_n^A(H\mid e)\approx K_n^A(H,e)-K_n^A(e)\\
&\approx K_n^A(H)+K_n^A(e\mid H)-K_n^A(e)\lessapprox \log|\H|+K_n^A(e\mid H)-K_n^A(e).
\end{align*}
Since $H$ gives a direction, $H^{\perp}$, in $e^{\perp}$ with precision $n$, we can find $e^\perp$ (and hence also $e$) with precision $n$ by giving at most $(d-2)n$ extra information. This shows that $K_n^A(e\mid H)\lessapprox (d-2)n$. Also, for most directions $e$ we have $K_n^A(e)\ge (d-1)n-\eps n$, hence
$$|\H'|\le |\H|2^{-n+\eps n +O(\log n)}=\delta^{1-O(\eps)}|\H|.$$
The set of those directions $e$ for which this fails is covered by $2^{(d-1)n-\eps n}=\delta^{-(d-1)+\eps}$  many $\delta$-cubes. Therefore we proved the following lemma.

\begin{lemma}\label{lem:discretizedIntersection}
    Let $d \geq 2$, $\delta=2^{-n}$ be a dyadic scale, and let $\eps > 0$. Let $\H$ be a set of $\delta$-hyperplanes in $\R^d$, and let $A$ be an oracle relative to which $\H$ is computable.  Let $e \in \S^{d-1}$ with
    \begin{equation*}
        K^A_n(e) \geq (d-1)n - \eps n.
    \end{equation*}
    Then
    \begin{equation*}
        \H':= \{H\in \H\mid H^\perp \subset e^\perp\}
    \end{equation*} 
    satisfies 
    \begin{equation*}
        \vert \H'\vert \lesssim \delta^{1-O(\eps)} \vert \H \vert.
    \end{equation*}
\end{lemma}

We now turn to a discrete analog of Mattila's slicing theorem.
\begin{theorem}[Mattila \cite{Mattila75}]\label{thm:mattilaslicing}
    Let $d - k < s \leq d$, and let $A\subseteq\R^d$ be an $\mathcal{H}^s$-measurable set with $0 < \mathcal{H}^s(A)< \infty$. Then, for almost all $V\in G(d,k)$,
    \begin{equation*}
        \mathcal{H}^{d-k}(\{u\in V^\perp\mid \dim_H(A\cap (V+u)) = s+ k -d\}) > 0.
    \end{equation*}
\end{theorem}

In the special case when $k=d-1$, this shows that for every $s>1$, for almost every direction $e\in S^{d-1}$, an $s$-set in $\R^d$ intersects positively many shifts of the hyperplane $e^\perp$ in a set of dimension $s-1$. 

Our proof of a discretized analog of this result uses algorithmic methods, in particular, the results we presented in Section 2.

\begin{theorem}\label{thm:discretizedSlicing}
    Let $d \geq 2$, $1 <s\leq d$, and $\eps, \eta \in (0,s-1)$. For all sufficiently small dyadic scales $\delta$, the following holds. Let $\P$ be a $(\delta, s, \delta^{-\eta})$-set of $\delta$-cubes in $[0,1]^d$. Let $A$ be an oracle relative to which $\P$, $\delta = 2^{-n}$, $s$, $\eps$ and $\eta$ are computable. Then, for every $e \in \S^{d-1}$, with 
    \begin{equation*}
        K^A_n(e) \geq (d-1)n - 2\eps n,
    \end{equation*}
    there are $\gtrsim\delta^{-1 + O(\eps + \eta)}$ many dyadic $\delta$-intervals $I$ in $\R$ for which
    \begin{equation*}
        \P_1:= \{Q \in \P \mid Q \cap p_e^{-1}I \neq \emptyset\}
    \end{equation*}
    contains a $(\delta, s-1, \delta^{-O(\eps+\eta)})$-subset $\P_2$ with $\vert \P_2 \vert \gtrsim \delta^{1+ o(1)}\vert \P\vert $. 
\end{theorem}
\begin{proof}
     Let $e\in S^{d-1}$ such that
    \begin{equation*}
        K^A_n(e) \geq (d-1)n - 2\eps n.
    \end{equation*}
    Note that this implies that, for every $m \leq n$,
    $$K^A_m(e) \geq (d-1)m - 2\eps n - O(\log n).$$
    
    Let $Q\in\P$ such that $K^A_n(Q) \approx \log \vert \P\vert$, $K^A_m(Q) \geq sm - \eta n$ for every $m \leq n$ and such that $K^{A,e}_m(Q) \approx K^A_m(Q)$ for every $m\leq n$. By Proposition \ref{prop:sSetToComplexities}, such a $Q$ exists. Let $I$ be a dyadic $\delta$-interval meeting $p_e Q$. 

    Using the fact that  $I$ can be encoded into a finite string, and therefore putting it into an oracle is the same as using conditional complexities, and using symmetry of information, we can see that for every $m\leq n$, 
    \begin{align*}
        K^{A,e,I}_{m}(Q) &\approx K^{A,e}_{m,n}(Q\mid I)
        \approx K^{A,e}_n(Q\mid I) - K^{A,e}_{n,n,m}(Q\mid I, Q).\\
        \end{align*}
        For $K^{A,e}_n(Q\mid I)$ we use the trivial estimate $$K^{A,e}_n(Q\mid I)\approx K^{A,e}_n(Q) -K_n^{A,e}(I)\gtrapprox  K^{A,e}_n(Q) - n.$$Note that this immediately implies that
        \begin{equation}\label{eq:QgivenQe}
            K^{A,e}_n(Q\mid I) \gtrapprox \log |\P| - n.
        \end{equation}
        
        For estimating $K^{A,e}_{n,n,m}(Q\mid I, Q)$, note that by our assumptions on $e$ and $Q$, we can apply Lemma \ref{lem:algSlicing}. Applying the second part of the lemma we obtain
        $$K^{A,e}_{n,n,m}(Q\mid I, Q)\lessapprox K^A_n(Q) - n -(s-1)m +O((\eps +\eta) n).$$ Therefore
\begin{align*}
 K^{A,e,I}_{m}(Q)       &\gtrapprox K^{A,e}_n(Q) - n - (K^A_n(Q) - n -(s-1)m+O((\eps +\eta) n))\\
        &=(s-1)m - O((\eps +\eta) n).
    \end{align*}

Note that     \begin{equation*}
        \P_1 = \{R\in \P \mid R \cap p_e^{-1} I \neq \emptyset\}
    \end{equation*}
is enumerable given $A$, $e$ and $I$, and moreover, $\P_1$ contains $Q$. 

Therefore, by Proposition \ref{prop:enumerableSset}, the set $\P_1$ contains a $(\delta, s-1, \delta^{-O(\eps+\eta)})$-subset $\P_2$ with 
$$|\P_2|\ge 2^{K_n^{A,e,I}(Q)-O((\eps+\eta)^{-1}\log n)}.$$ Using \eqref{eq:QgivenQe}, this shows $|\P_2|\ge \delta^{1+o(1)}|\P|$. 

What remains to show is that the same estimates hold not only with our $I$, but also for many other dyadic intervals in $\R$. 

Given $A$ and $e$, the set of all dyadic intervals $I$ for which there is a set $\P_2$ that satisfies the requirements of our theorem, is enumerable. Therefore the number of dyadic intervals for which there is such a $\P_2$ is $\gtrsim 2^{K^{A,e}(I)}$.

From the second statement in Lemma \ref{lem:algSlicing}, with $m = 0$, we get $$K_n^{A,e}(Q\mid I)\lessapprox K_n^{A,e}(Q)-n+O((\eps+\eta) n).$$
Since given $e$ and $Q$ we also know $I$, therefore in this formula we have $K_n^{A,e}(Q)\approx K_n^{A,e}(Q,I)$, and then by the symmetry of information we get
    \begin{equation*}\label{eq:boundQe}
        K^{A, e}_n(I) \gtrapprox n - O((\eta + \eps)n).
    \end{equation*}
    Therefore indeed there are $\gtrsim 2^{ n - O((\eta + \eps)n)}=\delta^{-1+O(\eta + \eps)}$ many dyadic intervals $I$ for which the statement of our theorem holds.

\end{proof}

Our next lemma is a discretized analog of Marstrand's projection theorem. 
\begin{theorem}\label{thm:discretizedProjection}
    Let $t>0,\eps, \eta > 0$. For all sufficiently small dyadic scales $\delta$, the following holds. Let $\P$ be a $(\delta, t, \delta^{-\eta})$-set of $\delta$-cubes in $[0,1]^d$. Let $A$ be an oracle relative to which $\P$, $\delta = 2^{-n}$, $t$, $\eps$ and $\eta$ are computable. Then, for every $e \in \S^{d-1}$, with 
    \begin{equation*}
        K^A_n(e) \geq (d-1)n - \eps n,
    \end{equation*}
    the set 
    \begin{equation*}
        \P_1: = \D_\delta(\bigcup_{Q\in\P}\pi_{e\perp} Q)
    \end{equation*}
     contains a $(\delta, t^\prime, \delta^{-O(\eps+\eta)})$-subset $\P_2$ in $\R^{d-1}$ with $t^\prime = \min\{t, d-1\}$. Moreover,
     \begin{equation*}
         |\P_2| \gtrsim \delta^{-t^\prime + O(\eps + \eta)}
     \end{equation*}
\end{theorem}
\begin{proof}
    As before, we encode $\P$, $\delta = 2^{-n}$, $t$, $\eps$ and $\eta$ into an oracle $A$, and first choose $e$ such that $K^A_n(e) \geq (d-1)n - \eps n$, and then choose $Q\in\P$ such that $K^A_n(Q) \approx \log \vert \P\vert$, $K^A_m(Q) \gtrapprox tm - \eta n$ for every $m \leq n$ and such that $K^{A,e}_m(Q) \approx K^{A}_m(Q)$ for every $m\leq n$. From the first part of Lemma \ref{lem:algSlicing}, we get 
    \begin{equation*}
        K^{A,e}_m(\pi_{e^\perp} Q) \gtrapprox\min\{tm, (d-1)m\} - O((\eps+\eta) n),
    \end{equation*}
    and the conclusion follows from Proposition \ref{prop:enumerableSset}.
\end{proof}

\section{Incidence bounds for hyperplanes and hyperplane Furstenberg sets}\label{sec:IncidenceHyperplanes}
In this section, we prove Theorem \ref{thm:hyperplanesIncidence}, our main result on hyperplane incidence bounds. In Section \ref{ssec:incidenceHyperplaneAlgorithmic}, we will reformulate this bound in the language of complexities. This will allow us to prove our two applications of Theorem \ref{thm:hyperplanesIncidence}, to hyperplane Furstenberg sets and to pinned distance sets. 

\bigskip

\subsection{Discretized version}\label{ssec:discretizedworld}

We will prove Theorem \ref{thm:hyperplanesIncidence} by induction on the dimension of the space $\R^d$, with Theorem \ref{thm:RenWang} as a base case. Our inductive step will rely on Lemma \ref{lem:inductionHelpPlanes2} below.
In order to simplify the statement of our lemma, we introduce the following definition.

\medskip
We say that a pair $(\P,\H) \subset \D_\delta([0,1]^d) \times \H^\delta$ is a $(\delta, s, \delta^{-\eta}, M)$-nice configuration, if for each $Q\in P$ there are $\sim M$ many $\delta$-hyperplanes in $\H$ through $Q$, with the property that the set of directions of these $\delta$-hyperplanes is a $(\delta, s, \delta^{-\eta})$-set. (More precisely, for each $H^\perp$ we fix a $\delta$-cube that meets $H^\perp$, and we require that through each $Q\in\P$ there are $M$ different $\delta$-hyperplanes in $\H$ such that the $\delta$-cubes of the normal vectors of these $\delta$-hyperplanes are distinct and their union is a $(\delta, s, \delta^{-\eta})$-set of $\delta$-cubes.)

\medskip
We often find it useful to identify the set of directions in $\R^d$, i.e. the sphere $S^{d-1}$, by $\R^{d-1}$, via radial projections. This can be done, say, by fixing an affine hyperplane $H$ at distance 1 from the origin, and then identifying each direction by the point at which the line through the origin of this direction hits the plane $H$. We will refer to this mapping as the 'radial projection onto $H$'.

A technical issue that we need to be careful about is that not every line meets $H$ (those whose direction is parallel to $H$ don't), and moreover, although this mapping from the sphere to $H$ is locally bi-Lipschitz, and therefore it preserves many nice properties, unfortunately it gets distorted as the directions of the lines are getting close to being parallel to $H$. Therefore, it is \emph{not} true that the radial image of a $(\delta, s, \delta^{-\eta})$-set of directions is a $(\delta, s, \delta^{-\eta})$-set on $H$.

However, by throwing away a constant portion of our $(\delta, s, \delta^{-\eta})$-set of directions, we can find a plane $H$ such that each remaining direction has at least a constant angle with this plane $H$ and therefore our mapping will be bi-Lipschitz with an absolute constant. Noting also that a constant portion of a   $(\delta, s, \delta^{-\eta})$-set is a $(\delta, s, c\delta^{-\eta})$-set, this argument shows that for each
 $(\delta, s, \delta^{-\eta})$-set of directions there is a plane $H$ such that the radial projection onto this plane $H$ contains a $(\delta, s, c\delta^{-\eta})$-set of comparable size. Moreover, we can find not only one $H$ with this property but this will be true for a constant portion of all affine hyperplanes $H$ (among those whose distance from the origin is 1). Every constant in this observation is an absolute constant that depends only on the dimension $d$. 
 
 We will make use of this observation in the proof of the following lemma.

\begin{lemma}\label{lem:inductionHelpPlanes2}
    For every $s \in (1, d-1]$, $t\in (0,d]$, and $\eta > 0$ there is a $\delta_0$ such that the following holds for every $\delta < \delta_0$. If $(\mathcal{P}, \H)$ is a $(\delta, s, \delta^{-\eta}, M)$-nice configuration and $\P$ is a $(\delta, t, \delta^{-\eta})$-set, then there is a set of $\delta$-cubes $\P^\prime$ in $[0,1]^{d-1}$ and a set of $\delta$-hyperplanes $\H'$ in $\R^{d-1}$ such that
    \begin{enumerate}
        \item[(i)] $(\P^\prime, \H^\prime)$ is a $(\delta, s-1, \delta^{-O(\eta)}, \delta^{1+O(\eta)} M)$-nice configuration;
        \item[(ii)]$\P^\prime$ is a $(\delta, t^\prime, \delta^{-O(\eta)})$-set with $t^\prime = \min\{t, d-1\}$ and $|\P'|\gtrsim \delta^{-t'+O(\eta)}$; and 
        \item[(iii)]$\vert \H^\prime \vert \lesssim \delta^{1-O(\eta)}\vert \H \vert$.
    \end{enumerate}
\end{lemma}
\begin{proof}
We will show that there is a direction $e\in S^{d-1}$, such that with a (carefully chosen) subset $$\P'\subset\{\pi_{e^\perp} Q\mid Q\in\P\}$$ and with $$\H':=\{H\cap e^\perp\mid H\in \H, H^\perp\in e^\perp\}$$ the lemma holds (identifying $e^\perp$ with $\R^{d-1}$). 

{\bf Property (iii)}: By Lemma \ref{lem:discretizedIntersection}, (iii) holds with most $e\in S^{d-1}$. 
Our aim is to find $e$ and $\P'$ such that also (i) and (ii) hold.

{\bf Property (i)}: We encode $\P$, $\H$, $\delta = 2^{-n}$, $s$, $t$, and $\eta$ into an oracle $A$. In addition, we encode, for every $Q \in \P$, a witnessing family of $\delta$-hyperplanes $\H(Q)\subset\H$ through the cube $Q$ given by our nice definition. We fix a cube $Q \in \P$ such that $K^A_n(Q) \approx \log |\P|$, and for all $m\leq n$ we have $K^A_m(Q) \gtrapprox tm - \eta n$. By Proposition \ref{prop:sSetToComplexities}  such a $Q$ exists.    

By our discussion at the beginning of this section, we know that there is an affine hyperplane $H_0$ in distance 1 from the origin such that the radial projection onto $H_0$ of the set of normal vectors of the $\delta$-hyperplanes in $H(Q)$ contains a $(\delta, s, c\delta^{-\eta})$-set of $\sim M$ many $\delta$-cubes in $H_0$. Applying our slicing result, Theorem \ref{thm:discretizedSlicing} to this set of $\delta$-cubes in $H_0$, we can see that most 1-codimensional affine subspaces in $H_0$ (that is, for most $e\in S^{d-1}$, the affine hyperplane $H_0\cap e^\perp$) contains a $(\delta,s-1,\delta^{-O(\eta})$-subset of $\gtrsim \delta^{1+O(\eta)}M$ many $\delta$-cubes. In particular, for most directions $e\in S^{d-1}$, the set
$$\{H\in H(Q)\mid H^\perp\subset e^\perp\}$$ contains a $(\delta,s-1,\delta^{-O(\eta)})$-subset of $\gtrsim \delta^{1+O(\eta)}M$ many $\delta$-hyperplanes. This immediately implies that we have a $(\delta,s-1,\delta^{-O(\eta)})$-subset of $\gtrsim \delta^{1+O(\eta)}M$ many $\delta$-hyperplanes in $\H'$ through $\pi_{e^\perp}Q$. 

In order to satisfy (i), we need this to hold not only for $\pi_{e^\perp}Q$ with our fixed cube $Q$ but for every cube in $\P'$. Therefore, let $\P_1$ denote the set of all cubes $R\in P$ for which $$\{H\in H(R)\mid H^\perp\subset e^\perp\}$$ contains a $(\delta,s-1,\delta^{-O(\eta})$-subset of $\gtrsim \delta^{1+O(\eta)}M$ many $\delta$-hyperplanes. We will choose $\P'\subset\pi_{e^\perp}\P_1$, and therefore (i) will hold. What remains is to check that there is an $e$ for which $\pi_{e^\perp} \P_1$ contains a set $\P'$ for which (ii) holds. 

{\bf Property (ii)}: We have already seen that, for most $e$, we have (iii) and $Q\in\P_1$. We fix an $e$ such that indeed (iii) holds, $Q\in\P_1$, and moreover, we choose $e$ such that 
\begin{equation}\label{eq:choiceOfe}
K_n^A(e\mid Q)\ge (d-1)n-O(\eta n).
\end{equation}
We may choose $e$ so that it is computable given its first $n$ bits. Note that, since conditioning on $Q$ can only decrease the complexity of $e$, \eqref{eq:choiceOfe} immediately implies that $K^A_n(e) \geq (d-1) n - O(\eta n)$. By symmetry of information, for any $m \leq n$, 
\begin{align*}
    K^A_n(e) &\approx K^A_{m}(e) + K^A_{n,m}(e\mid e)\\
    &\lessapprox K^A_m(e) + (d-1)(n-m),
\end{align*}
and so $K^A_m(e) \gtrapprox (d-1)m - O(\eta n)$. In addition, \eqref{eq:choiceOfe} shows that $Q$ gives little information about $e$. So, by symmetry of information, $e$ gives little information about $Q$, in particular,
\begin{equation}
    K^A_n(Q \mid e) \geq K^A_n(Q) - O(\eta n).
\end{equation}
Thus, we can see that (i)-(iii) of Lemma \ref{lem:algSlicing} hold with $s=t$, $\eps=O(\eta)$ and $x=Q$. 
From Lemma \ref{lem:algSlicing} we get
\begin{equation}\label{eq:boundProjQ}
    K^{A,e}_m(\pi_{e^\perp} Q) \gtrapprox \min\{t, d-1\}m - O(\eta n)
\end{equation}
for every $m\leq n$. And since $\pi_{e^\perp} Q$ is in $\pi_{e^\perp} \P_1$ and  $\pi_{e^\perp} \P_1$ is enumerable, by \eqref{eq:boundProjQ} and Proposition \ref{prop:enumerableSset} it has a 
$(\delta, t^\prime, \delta^{-O(\eta)})$-subset $\P'$ with 
\begin{equation}
    |\P'|\geq 2^{K^{A,e}_n(\pi_{e^\perp}(Q)) - O(\eta^{-1}\log n)} \gtrsim \delta^{-t^\prime + O(\eta)}.
\end{equation}
\end{proof}

We are now able to prove our main theorem.
\begin{proof}[Proof of Theorem \ref{thm:hyperplanesIncidence}]
As we already mentioned it in the introduction, the planar case is a famous result of Ren and Wang. By induction, we now assume that it is true in $\R^{d-1}$, and we check that it holds in $\R^d$. 

Let $(\P, \H)$ be a $(\delta, s, \delta^{-\eta}, M)$-nice configuration in $\R^{d}$ and $\P$ be a $(\delta, t, \delta^{-\eta})$-set. Assume that $\delta$ is small enough so that we can apply 
Lemma \ref{lem:inductionHelpPlanes2} to $(\P, \H)$. Then the lemma gives us $(\P^\prime, \H^\prime)$ in $\R^{d-1}$. We also assume that $\eta$ and $\delta$ are small enough so that we can apply
 Theorem \ref{thm:hyperplanesIncidence} to  $(\P^\prime, \H^\prime)$, more precisely, we assume that we can apply Theorem \ref{thm:hyperplanesIncidence} in $\R^{d-1}$ with $\eps/2$, $s-1$, $t'$, $c\eta$ and $\delta$ (where $c$ is the constant in the $O(\eta)$ terms in the statement of Lemma \ref{lem:inductionHelpPlanes2}). For simplicity, we also assume that $\eta<\eps/2$. Then, 
since $|\H| \gtrsim \delta^{-1+c\eta}|\H^\prime|$ and $|\H'|\gtrsim\delta^{-\min\{t^\prime, \frac{s - 1 +t^\prime - ((d-1)-2)}{2}, 1\} + \eps/2}\delta M$, therefore

$$ |\H|\gtrsim \delta^{-\min\{t^\prime, \frac{s +t^\prime - (d-2)}{2}, 1\} + \eps} M.$$

    If $t'=t$, then the conclusion follows immediately. And with $t'=d-1$ we get $$\min \{t', \frac{s +t^\prime - (d-2)}{2}, 1\}=1=\min \{t, \frac{s +t - (d-2)}{2}, 1\}$$ 
and the conclusion again follows immediately.
\end{proof}

\bigskip

\subsection{Algorithmic incidences}\label{ssec:incidenceHyperplaneAlgorithmic}
In this section, we use Theorem \ref{thm:hyperplanesIncidence} to establish bounds in the language of Kolmogorov complexity. 

The statement and the proof of our Theorem \ref{thm:algorithmicIncidenceHyperplanes} below is very similar to Theorem 3.5 of \cite{CsoStu25}, where we used the planar version of Theorem \ref{thm:hyperplanesIncidence}. 

\begin{theorem}\label{thm:algorithmicIncidenceHyperplanes}
    For every $s, t,\eps > 0$, with $ d-2<s\leq d-1$, $0<t\leq d$ there is an $\eta > 0$ and $n_0 \in \N$ such that the following holds for all $n \geq n_0$ and every oracle $A$. For every $y\in [0,1]^d$ and affine hyperplane $H$ in $\R^d$ containing $y$, if 
    \begin{enumerate}
        \item[(i)] $K^A_{k}(y) > tk - \eta n$ for all $k \leq n$, and
        \item[(ii)] $K^A_{k,n}(H \mid y) > sk - \eta n$ for all $k \leq n$
    \end{enumerate}
    then
    \begin{equation*}
        K^A_{n}(H) \gtrapprox K^A_{n}(H\mid y) + \alpha n - \eps n
    \end{equation*}
    with $\alpha: = \min\{t, \frac{s+t -(d-2)}{2}, 1\}.$
\end{theorem}
\begin{proof}
    Let $\eta$ be the parameter given by Theorem \ref{thm:hyperplanesIncidence}, and we choose $n_0$ large enough so that $2^{-n_0} < \delta$. Let $Q_n$ denote the dyadic cube containing $y$ with side-length $2^{-n}$. Define 
    \begin{equation*}
        \mathcal{H} = \{H^\prime \in \mathcal{H}^{2^{-n}} \mid K^A(H^\prime) \lessapprox K^A_{n}(H)\},
    \end{equation*}
    and for every dyadic cube $Q$, denote
    \begin{equation*}
        n_Q = \#\{H^\prime \in \mathcal{H} \mid H^\prime \cap Q \neq \emptyset\}.
    \end{equation*}
    Since $\mathcal{H}$ is enumerable given $K^A_{n}(H)$, and the $2^{-n}$-hyperplane containing $H$ is in $\mathcal{H}$, we have
    \begin{equation}\label{eq:sizeOfH}
        \vert \mathcal{H} \vert = 2^{K^A_{n}(H) + O(\log n)}.
    \end{equation}
 
    We note that
    \begin{equation}
        K^A_{n}(H\mid y) \lessapprox \log n_{Q_n} + O(\log n).
    \end{equation}

    Define the set 
$$\P_1 = \{Q\in \D_{2^{-n}}[0,1]^d \mid K^A(Q) \lessapprox K^A_{n}(y), n_Q \geq 2^{K^A_n(H\mid y) - O(\log n)}\}.$$
Then $Q_n\in\P_1$, and $\P_1$ is enumerable, given $n$, $K^A_n(H)$, $K^A_{n}(y)$ and  $K^A_n(H\mid y)$. 

We fix a rational $\eta_0$ which will be specified later. Define the set $\P$ to be the set of those $Q\in \P_1$ that satisfy the following.  There is a dyadic $(2^{-n}, s, 2^{O(\eta_0 n)})$-set $H(Q)\subseteq \H$ with 
 \begin{equation*}
     \vert H(Q)\vert \geq 2^{K^A_n(H\mid y) - O(\eta_0^{-1}\log n)}
 \end{equation*}
 such that $H \cap Q \neq \emptyset$ for all $H\in H(Q)$. The constants in the $O(\cdot)$ terms will come from Proposition \ref{prop:enumerableSset}.

The set $H(Q(y)) = \{H^\prime \in \H \mid H^\prime \cap Q(y) \neq \emptyset\}$ is enumerable given $K^A_{n}(H)$, and contains the $2^{-n}$-hyperplane containing $H$. Therefore, by the dual of Proposition \ref{prop:enumerableSset}, applied to $H(Q(y))$, we see that $H(Q(y))$ contains a subset $H^\prime$ which is a dyadic $(2^{-n}, s, 2^{O(\eta_0 n)})$-set and such that 
\begin{equation*}
    \vert H^\prime \vert \geq 2^{K^A_{n}(H\mid y) - O(\eta_0^{-1} \log n)}.
\end{equation*}

 Therefore we see that $Q(y) \in \P$. Furthermore, $\P$ is computably enumerable, given $s$, $\eta_0$, $K^A_{n}(y)$, $K^A_{n}(H)$ and $K^A_n(H\mid y)$. Since $K(\eta_0,s)=O(1)$, and the other inputs have length $O(\log n)$, by applying Proposition \ref{prop:enumerableSset} to $\P$, we obtain a set $\P^\prime \subseteq\P$ such that $\vert \P^\prime\vert \geq 2^{K^A_{n}(y) - O(\eta_0^{-1} \log n)}$ and $\P^\prime$ is a dyadic $(2^{-n}, t, 2^{O(\eta_0 n)})$-set. 

We now choose $\eta_0$ small enough so that we may apply Theorem \ref{thm:hyperplanesIncidence} to $\P^\prime$ with parameters $s$, $t$ and $\eps$ and $\eta:=(\log C')/n\sim O(\eta_0)$. Applying Theorem \ref{thm:hyperplanesIncidence} yields 
\begin{equation*}
    \vert \H \vert \gtrsim_\eps 2^{\min\{t, \frac{s+t -(d-2)}{2}, 1\}n - \eps n}2^{K^A_n(H\mid y) - O(\eta^{-1}\log n)}.
\end{equation*}
In particular, by \eqref{eq:sizeOfH}, we see that
\begin{align*}
    K^A_{n}(H) &\gtrapprox K^A_{n}(H\mid y) +\min\{t, \frac{s+t -(d-2)}{2}, 1\}n\\
    &\;\;\;\;\;\;\;\;\;\;\;  - \eps n - O(\eta^{-1}\log n),
\end{align*}
and the conclusion follows.
\end{proof}

By rescaling, we immediately have the following corollary.
\begin{corollary}\label{cor:algorithmicIncidenceHyperplanes2}
    For every $s, t,\eps > 0$, with $ d-2<s\leq d-1$, $0<t\leq d$ there is an $\eta > 0$ and $n_0 \in \N$ such that the following holds for all $n,m$ with $n-m \geq n_0$ and for every oracle $A$. For every $y\in [0,1]^d$ and affine hyperplane $H$ in $\R^d$ containing $y$, if
    \begin{enumerate}
        \item[(i)] $K^A_{k, m}(y\mid y) > t(k-m) - \eta (n-m)$ for all $m\leq k \leq n$, and
        \item[(ii)] $K^A_{k,n}(H \mid y) > sk - \eta (n-m)$ for all $k \leq n-m$
    \end{enumerate}
    then
    \begin{equation*}
        K^A_{n,m}(H\mid y) \gtrapprox K^A_{n-m,n}(H\mid y) + \alpha(n-m) - \eps (n-m)
    \end{equation*}
    with $\alpha: = \min\{t, \frac{s+t -(d-2)}{2}, 1\}.$
\end{corollary}

We will need the following simple corollary of this theorem, which will be easier to apply in our application to the distance set problem.
\begin{corollary}\label{cor:algorithmicIncidenceProjection}
    For every $s, t, \eps > 0$, with $d-2 < s \leq d-1$, $0 < t\leq d$ there is an $\eta > 0$ and $n_0\in \N$ such that the following holds for all $n,m$ with $n-m \geq n_0$ and for every oracle $A$. For every $y\in\R^d$ and $e \in \mathcal{S}^{d-1}$, if
    \begin{enumerate}
        \item[(i)] $K^A_{k, m}(y\mid y) > t(k-m) - \eta (n-m)$ for all $m\leq k \leq n$,
        \item[(ii)] $K^A_{k,n}(e \mid y) > sk - \eta (n-m)$ for all $k \leq n-m$, and
        \item[(iii)] $K^A_{n, n, n-m}(e\mid y,e) \approx 0$,
    \end{enumerate}
    then
    \begin{equation*}
        K^A_{n,n,n,m}(y\mid p_e y, e, y) \lessapprox K^A_{n,m}(y\mid y) - \alpha (n-m) + \eps (n-m),
    \end{equation*}
    where $\alpha = \min\{t, \frac{s+t -(d-2)}{2}, 1\}.$
\end{corollary} 
\begin{proof}
    Let $H$ be the affine hyperplane containing $y$ whose normal vector is $e$. Since knowing $e$ and $p_e y$ is the same as knowing $H$, our aim is to estimate $K_{n,n,m}(y\mid H, y)$. It also follows from the same observation that $K^A_{k,n}(e\mid y)\approx K^A_{k,n}(H\mid y)$ for every $k\le n$.
    
    By the symmetry of information, Corollary \ref{cor:algorithmicIncidenceHyperplanes2}, our observation, symmetry of information, again symmetry of information, and (iii), we have the following inequalities:
    $$K^A_{n,m}(H\mid y)+K^A_{n,n,m}(y\mid H,y)\approx K^A_{n,n,m}(y,H\mid y).$$ 
        $$K^A_{n,m}(H\mid y) \gtrapprox K^A_{n-m,n}(H\mid y) + \alpha(n-m) - \eps (n-m)$$
 $$K^A_{n-m,n}(H\mid y)\approx K^A_{n-m,n}(e\mid y)$$
    $$K^A_{n,n,m}(y,H\mid y)\approx K_{n,m}(y\mid y)+ K^A_{n}(H\mid y)$$ 
    
$$ K_n^A(H\mid y)\approx K^A_{n,n-m,n}(H,e\mid y)\approx K^A_{n-m,n}(e\mid y)+K^A_{n,n,n-m}(H\mid y,e)$$ 
$$K^A_{n,n,n-m}(H\mid y,e)\approx 0.$$
Putting all these together we get the statement of the corollary.
\end{proof}

\subsection{Hyperplane Furstenberg sets}
We are now able to prove our main theorem on hyperplane Furstenberg sets. 
\begin{proof}[Proof of Theorem \ref{thm:FurstenbergHyperplanes}]
    Let $A$ be a Hausdorff oracle for $F$ and for $\H$. Let $\eps > 0$. It suffices to show that 
    \begin{equation}
        \dim^A_H(x) \geq \min\{s+t, \frac{3s+t - (d-2)}{2}, s+1\} - \eps,
    \end{equation}
    for some $x\in F$.
    
    Let $H\in\H$ be a hyperplane such that $\dim^A_H(H) > t - \eps / 2$. Let $x \in F \cap H$ such that $\dim^{A,H}_H(x) > s - \eps / 2$. Fix rational $t^\prime > t - \eps$ such that $\dim^A_H(H) > t^\prime$ and rational $s^\prime > s - \eps$ such that $\dim^{A,H}_H(x) > s^\prime$. Let $n\in\N$ be sufficiently large and $\eta > 0$ be sufficiently small. 

    We now consider the dual of $x$ and $H$, let $H_x$ be the dual hyperplane of $x$ and $y_H$ be the dual point of $H$. Since $n$ is taken to be sufficiently large, we see that
    \begin{enumerate}
        \item $K^A_a(y_H) > t^\prime a - \eta n$ for all $a\leq n$, and
        \item $K^A_{b,n}(H_x\mid y_H) > s^\prime b - \eta n$ for all $b \leq n$.
    \end{enumerate}
    We may therefore apply Theorem \ref{thm:algorithmicIncidenceHyperplanes} (choosing $m = 0$), which yields
    \begin{equation}
        K^A_n(H_x) \gtrapprox K^A_n(H_x \mid y_H) + \alpha n - \eps n /2,
    \end{equation}
    where $\alpha = \min\{t^\prime, \frac{s^\prime + t^\prime - (d-2)}{2}, 1\}$.

    Hence, we see that 
    \begin{align*}
        K^A_n(x) &\approx K^A_n(H_x)\\
        &\gtrapprox s^\prime n - \eta n + \alpha n - \eps n /2\\
        &= \min\{s^\prime + t^\prime, \frac{3s^\prime + t^\prime - (d-2)}{2}, s^\prime + 1\} n - \eta n - \eps n / 2.
    \end{align*}

    Therefore $$\dim^A_H(x) \geq \min\{s^\prime + t^\prime, \frac{3s^\prime + t^\prime - (d-2)}{2}, s^\prime + 1\} - \eta - \eps / 2$$and the conclusion follows.
\end{proof}

\section{Partitioning Lipschitz Functions}\label{sec:LipschitzFunction}
In this section, we will introduce a method of partitioning a piecewise linear Lipschitz function $f:[0,n]\to \R$ which will be used in the proof of Theorem \ref{thm:mainthm1}. 

The ultimate goal of this section is to partition the complexity function $K_n(x)$ of a point $x\in \R^d$ (as a function of $n$) into intervals of precisions where we have some control over how the complexity of $x$ is changing. Of course, the complexity of $x$ is only defined on the natural numbers (precisions), but we can interpolate to make it piecewise linear. We will interpolate between precisions which are sufficiently far apart to make the function also Lipschitz. For more details on the construction of our function from the complexity, see the next section. 

\begin{definition} Let $f\colon\R\to\R$ be a continuous function. We say that $f$, on an interval $[a,b]$, is
\begin{itemize}
\item \emph{Frostman}, if it attains its minimum on $[a,b]$ at $a$;
\item \emph{Katz-Tao}, if it attains its minimum on $[a,b]$ at $b$. 
\end{itemize}
Since the function $f$ is typically fixed, we commonly refer to intervals themselves as Frostman or Katz-Tao.
\end{definition}

First we recall the following well-known fact about Lipschitz functions. It is (a special case of) Jones's traveling salesman theorem \cite{jones1990rectifiable}.
If $f\colon [0,1]\to\R$ is a Lipschitz function with Lipschitz constant at most $L$, where $L\ge 1$, then 
$$\sum_I|\beta(I)|^2|I|\lesssim L^2,$$
where the summation is taken for all dyadic subintervals of $[0,1]$. The so-called 'Jones $\beta$-number' of an interval $I$ is defined as $$\beta(I):= \frac{1}{|I|}\min_A\max_{t\in I} |f(t)-A(t)|,$$ where the minimum is taken for all affine functions $A$. 
 
An immediate corollary of this is that for every $\eps,\beta>0$ and $L>1$ there is a $\delta>0$, such that for every Lipschitz function $f\colon[0,1]\to\R$ with Lipschitz constant at most $L$, there are dyadic subintervals $I_j$ such that $|\bigcup_j I_j|\ge 1-\eps$, and $|I_j|\ge\delta$ and $\beta(I_j)\le \beta$ for every $j$. In particular, the following corollary holds.

\begin{corollary}\label{traveling}For every $\eps,\beta>0$ and $L\ge 1$ there is a $c\in\N$, such that for every Lipschitz function defined on an interval $I$ with Lipschitz constant at most $L$, there are at most $c$ non-overlapping subintervals in $I$ whose $\beta$-number is at most $\beta$ and whose total length is at least $(1-\eps)|I|$.
\end{corollary}

We will say that an interval $I=[a,b]$ has the 'doubling' property, if $b\le 2a$, and it has the 'non-doubling' property if $b\ge 2a$ (it can have both). Our main result in this section is the following lemma.
\begin{lemma}\label{parti}For every $\eps,\eta>0$ and $L\ge 1$ there is a $c>0$ such that the following holds. Let $f\colon [a,b]\to \R$ be a piecewise linear Lipschitz function with Lipschitz constant at most $L$, where $1\le a\le b$, and suppose that $f$ is Katz-Tao on $[a,b]$. Then we can find at most $c(1+\log(b/a))$ many non-overlapping doubling subintervals in $[a,b]$ with the following properties. The total length of these intervals is at least $(1-\eps)(b-a)$, and they can be divided  into three sets of intervals, which we denote by $G$, $R$ and $B$, with the following properties.
\begin{itemize}
\item[(i)] The intervals in $G$ are Frostman and Katz-Tao.
\item[(ii)] The intervals in $R$ are Frostman and they are not Katz-Tao.
\item[(iii)] The total length of the intervals in $G$ is at least the total length of the intervals in $R$.
\item[(iv)] The intervals in $B$ are Katz-Tao and their $\beta$-number is at most $\eta$.
\end{itemize}
\end{lemma}

\begin{proof}
First, using a greedy algorithm that always chooses the longest available interval, we choose a maximal set of non-overlapping doubling subintervals in $[a,b]$ with the property that $f$ is both Frostman and Katz-Tao on each of these subintervals. We call these intervals \emph{good} subintervals. Together with the remaining gaps, they form a partition of $[a,b]$. Within this partition, when we have a sequence of consecutive good subintervals, we call it a \emph{good block}. We have partitioned $[a,b]$ into good blocks and the gaps between the good blocks. 

Now consider a gap $[c,d]$. Since the function is piecewise linear, and it doesn't contain any good subinterval, therefore it cannot have a piece on which it is increasing followed by a piece on which it is decreasing. In other words, denoting by $e$ the point at which it attains its minimum, it has to be strictly decreasing on $[c,e]$ and strictly increasing on $[e,d]$. We divide $[c,d]$ into two subintervals at the point $e$ (if $e$ is one of the endpoints, then we don't divide it).
We now have a partition of $[a,b]$ into decreasing intervals, increasing intervals, and good blocks.  

Importantly, every increasing interval in this partition is followed by a good block (note that here we use that the last interval in the partition cannot be increasing since the function was Katz-Tao on the interval $[a,b]$). We call a maximal set of consecutive intervals in this partition that are increasing, good block, increasing, good block, ..., increasing (it needs both to start and to end with increasing) an \emph{increasing block}. The function may not be increasing on an increasing block, but it will be important for us that there is a small neighborhood of its right endpoint on which it is increasing. 

Also note that, by our construction, every increasing block is followed by a good block, and this good block is followed by a decreasing interval. (Again, this good block must be followed by something, otherwise $[a,b]$ wouldn't have been Katz-Tao, and it cannot be followed by an increasing interval.) Since, outside this good block, at its left and right endpoint, there are some intervals on which the function is increasing and decreasing, respectively, therefore we can cover our good block by a slightly longer interval which has the Frostman and the Katz-Tao property. Since our good intervals were maximal, therefore this slightly larger covering cannot be doubling. Therefore our good block has the non-doubling property.

Now consider our partitioning of $[a,b]$ into increasing blocks, and non-doubling good blocks (the good blocks immediately following the increasing blocks), and the remaining gaps, which we will call 'decreasing blocks'. Since every increasing block is followed by a non-doubling good block, and we can have $\le 1+\log(b/a)$ many non-overlapping non-doubling intervals, therefore this is a partition of $[a,b]$ into $O(1+\log(b/a))$ many blocks.

Next, by subdividing each increasing interval (if necessary), and then joining consecutive intervals (if necessary),  we repartition each increasing block into doubling intervals. We can ensure that the total number of these doubling intervals is $\lesssim 1+\log (b/a)$. Note that each resulting interval is Frostman. These intervals will be (some of) our intervals $I$ in the statement of the lemma. Depending on whether an interval is also Katz-Tao, we assign it either to the set $G$ or to $R$. Similarly, by joining consecutive intervals if necessary, we repartition each non-doubling good block, into a total number of  $\lesssim 1+\log (b/a)$ many doubling subintervals that are both Frostman and Katz-Tao, and we assign these intervals to $G$.

Now consider a decreasing block. We first decompose the block into doubling intervals in the natural way. By applying Corollary \ref{traveling} with $\beta:=\eps\eta$ to each interval, we can choose $c$ non-overlapping subintervals $I$ whose union covers at least a $(1-\eps)$ portion of the block, and the $\beta$-number of each of these intervals is at most $\eps\eta$. 

Recall that our decreasing block had been before subdivided into decreasing and good intervals. Fully covering some of these good intervals by an interval $I$ is fine, but we don't want our $I$ to contain only part of a good interval. So let $I'\subset I$ denote the largest subinterval of $I$ whose endpoints are not inside the interiors of the good intervals. We examine two cases. If $|I'|< \eps |I|$, then we throw away the interval $I$. And if $|I'|\ge \eps|I|$, then we keep $I'$ (and discard the rest of the interval $I$). Note that $I'$ is Katz-Tao, and $\beta(I')\le\beta(I)/\eps\le \beta/\eps=\eta$. We assign these $I'$ to our set $B$, and, in addition, we assign to $G$ each good interval whose interior contained some endpoints of some of our intervals $I$. This ensures that we threw away at most an $\eps$ portion of our block.

It only remains to check (iii). This follows from the fact that each interval in $R$ is contained in an increasing block, and each increasing block is followed by a non-doubling good block whose intervals are in $G$. And from the non-doubling property it follows that the length of the increasing block is at most the length of this non-doubling good block.
\end{proof}

We now rewrite our lemma into a form that will be convenient to apply to our complexity functions $K_n(x)$.  

\begin{definition}
We say that a function $f$ on an interval $I=[a,b]$ is 
\begin{itemize}
\item\emph{$\sigma$-Frostman} or \emph{$\sigma$-Katz-Tao}, if $f(t)-\sigma t$ is Frostman or Katz-Tao, respectively;
\item \emph{$(\sigma,\eta)$-linear}, if $|f(t)-f(a)-\sigma (t-a)|\le\eta (b-a)$.
\end{itemize}
\end{definition}

The following proposition is an immediate corollary of Lemma \ref{parti}.
\begin{proposition}\label{prop:existenceNicePartition}For every $\eps,\eta>0$ and $L\ge 1$ there is an $n_0$ such that the following holds for every $n\ge n_0$. Let $f\colon [0,n]\to \R$ be a monotone increasing piecewise linear Lipschitz function with Lipschitz constant at most $L$, and suppose that $f$ is $\sigma$-Katz-Tao on $[0,n]$. Then we can find $\lesssim \log n$ many non-overlapping doubling subintervals in $[0,n]$ such that each of them has length  $\ge C\eta^{-1}\log n$, the total length of these intervals is at least $(1-O(\eps))n$, and they can be divided  into three sets of intervals, which we denote by $G$, $R$ and $B$, with the following properties.
\begin{itemize}
\item[(i)] The intervals in $G$ are $\sigma$-Frostman and $\sigma$-Katz-Tao.
\item[(ii)] The intervals in $R$ are $\sigma$-Frostman and they are not $\sigma$-Katz-Tao.
\item[(iii)] The total length of the intervals in $G$ is at least the total length of the intervals in $R$, up to a $O(\eps n)$ error. That is, $$\text{length}(G) \geq \text{length}(R) - O(\eps n).$$
\item[(iv)] Each interval in $B$ is $(\sigma',\eta)$-linear with some $0\le \sigma'\le \sigma$ and also $\sigma$-Katz-Tao.
\end{itemize}
\end{proposition}

\begin{proof}We choose $n$ large enough so then the total length of at most $c(1+\log n)$ many intervals of length $\le C \eta^{-1}\log n$ is at most $O(\eps n)$. Then all assertions are satisfied.
\end{proof}

\begin{remark} To take away some of the mystery (or, to make this sound even more mysterious), we note that the 'doubling property' of our intervals will be needed because of the curvature of the circle. This will become apparent in the geometric argument in the proof of Proposition \ref{prop:conditionalBoundsFrostmanIntervals}.
\end{remark}

We now introduce the complexity analogs of non-concentration conditions defined above. 
\begin{definition} Let $\sigma, \eta > 0$. For any integers $m \leq n$, and any point $x$, we say that $[m,n]$ is 
\begin{itemize}
\item \emph{$(\sigma, \eta)$-Frostman} if $K^A_k(x) - K^A_m(x) \gtrapprox \sigma(k-m) - \eta(n-m)$ for every $k\in[m,n]$;
\item \emph{$(\sigma, \eta)$-Katz-Tao} if $K^A_n(x) - K^A_k(x) \lessapprox \sigma(n-k) + \eta(n-m)$ for every $k\in[m,n]$;
\end{itemize}
\end{definition}

\begin{remark}\label{remark:complexityfunction}
    One of the important facts we will use is that in Proposition \ref{prop:existenceNicePartition}, the selected intervals have length at least $C\eta^{-1}\log n$. Therefore, if we approximate a complexity function by a piecewise linear function $f$ as described in Section \ref{ssec:prelimAlgMethods1}, if $f$ is $\sigma$-Frostman (respectively, $\sigma$-Katz-Tao) on $[a,b]$, then the complexity function is $(\sigma, \eta)$-Frostman (respectively, $(\sigma, \eta)$-Katz-Tao) on $[a,b]$.
\end{remark}

\begin{remark}\label{remark:complexityfunction2}
We also note the trivial fact that in our piecewise linear approximation, without loss of generality we can assume that $f$ is not only Lipschitz with Lipschitz constant at most $d+1$, but also 
    \begin{equation}\label{flip}(f(b)-f(a))/(b-a)\le d+c(\eta)\end{equation} 
    for every $a,b$ with $b-a\gtrsim C\eta^{-1}\log n$, where $c(\eta)$ is a constant depending on $\eta$ and it is small if $\eta$ is small.
\end{remark}

\section{Algorithmic dimension of distances}\label{sec:algorithmicDimDistances}
The goal of this section is to prove Theorem \ref{thm:maintheoremEff}. Throughout this section, we fix an $x$, $y$, and $e = \frac{y-x}{|x-y|}$ for which the following assumptions hold:
\begin{itemize}
\item[(i)] $\dim_H^A(x), \dim^A_H(y), \dim^{A,x}_H(e) >\frac{d}{2}$ and $\dim_H^A(x), \dim^A_H(y), \dim^{A,x}_H(e)>s_0>d-2$
\item[(ii)] $K^{A, x}_n(y) \gtrapprox K^{A}_n(y)$ for every $n\in\N$.
\end{itemize}
Without loss of generality we also assume that $|x-y| \sim 1$.

Our first lemma shows that we necessarily have a strong lower bound on the complexity of $e$ given $y$.
\begin{lemma}\label{lem:C1C2C3}
    There is an $s > d/2$ such that for all $n$ and all $m\leq n$,
\begin{equation}
    K^A_{m, n}(e \mid y) \gtrapprox sm.
\end{equation}
\end{lemma}
\begin{proof}
    Let $n$ be sufficiently large, and $m \leq n$. Using symmetry of information, and assumption (ii), and the obvious fact $\pi_{e^\perp} x = \pi_{e^\perp} y$, we have
    \begin{align*}
        K^A_m(x) &\approx K^{A}_{m,n}(x\mid y)
        \lessapprox K^A_{m,n}(e \mid y) + K^A_{m,m,n}(x \mid e,y)\\
        &\lessapprox K^A_{m,n}(e \mid y) + K^A_{m}(x \mid \pi_{e^\perp} x, e).
    \end{align*}

    By Lemma \ref{lem:algProjection} and assumption (i), there is an $s^\prime > d/2$ such that
    \begin{equation}
        K^A_m(x \mid \pi_{e^\perp} x, e) \lessapprox K^A_m(x) - s^\prime m + O(\eps m).
    \end{equation}
    Combining this with the previous inequality proves the lemma.
\end{proof}

\begin{proposition}\label{prop:conditionalBoundsFrostmanIntervals}
For every $\sigma\in[0,d]$ and $\eps > 0$ there is an $\eta>0$ such that the following holds for every doubling interval $[m,n]$ with $n-m\ge C \eta^{-1} \log n$. If $[m,n]$ is $(\sigma, \eta)$-Frostman, then
\begin{equation*}
    K^{A,x}_{n,n,m}(y\mid |x-y|, y) \lessapprox K^A_{n,m}(y\mid y) - \alpha (n-m) + \eps (n-m),
\end{equation*}
where $\alpha = \min\{\sigma, \frac{\frac{d}{2}+\sigma -(d-2)}{2}, 1\}.$
\end{proposition}
\begin{proof}
The proof of this proposition relies on the following elementary observation. If the angle between $y\in\R^d$ and $e\in S^{d-1}$ is $\theta$, then $p_e y=|y|\cos \theta$. In this formula, if we know $\theta$ with precision $m$, and $\theta\lessim 2^{-(n-m)}$, then we know $\cos\theta$ with precision $n$. If, in addition, $|y|\sim 1$ and we know $y$ with precision $n$, then we know $p_e y$ with precision $n$.

We apply this observation with $y$ replaced by $x-y$ and $e$ replaced by the $2^{-(n-m)}$-dyadic representative (defined in Section \ref{ssec:prelimAlgMethods1}) of the direction between $x$ and $y$. If we know $x$, and we know $y$ with precision $m$, then we know the direction between $x$ and $y$ with precision $m$. Since $m\ge n-m$, therefore we know $e'$. If, in addition, we know $|x-y|$ with precision $n$, then we know $p_{e'} (x-y)$ with precision $n$ and hence we know $p_{e'}y$ with precision $n$. 

Since if we know $e'$ with precision $n-m$ then we know $e'$ with every precision, therefore this shows that

     $$K^{A,x}_{n,n,m}(y\mid |x-y|, y) \lessapprox K^{A}_{n,n,n,m}(y \mid p_{e^\prime}y, e^\prime, y),$$ 
     and it also shows that (iii) of Corollary \ref{cor:algorithmicIncidenceProjection} holds with $e$ replaced by $e'$.
     
     Also, assumption (i) of  Corollary \ref{cor:algorithmicIncidenceProjection} holds with $t$ replaced by $\sigma$, and (ii) holds with $s$ given by Lemma \ref{lem:C1C2C3}. The proof is finished by applying 
     Corollary \ref{cor:algorithmicIncidenceProjection}.
\end{proof}

\begin{corollary}\label{cor:conditionalBoundsExactFrostman}
Let $m \leq n$ such that $n\leq 2m$ and $n-m\ge C \eta^{-1} \log n$. If $[m,n]$ is $(\sigma, \eta)$-Frostman with $\sigma \geq d/2$, then
    \begin{equation}\label{eq:boundSigmaGreaterDover2}
        K^{A,x}_{n,m}(|x-y|\mid |x-y|) \gtrapprox n-m - \eps (n-m).
    \end{equation}
    If $[m,n]$ is $(\sigma, \eta)$-Frostman and $(\sigma, \eta)$-Katz-Tao with $\sigma < d/2$,
    then
    \begin{equation}\label{eq:boundSigmaLessThanDover2}
        K^{A,x}_{n,m}(|x-y|\mid |x-y|) \gtrapprox K^A_{n,m}(y\mid y)/2 - \eps (n-m).
    \end{equation}
\end{corollary}
\begin{proof}
    We first note that, since $x$ doesn't give any information about $y$, we have
    \begin{align*}
        K^{A,x}_{n,m}(y\mid y) &\approx K^{A,x}_{n}(y) - K^{A,x}_m(y)
        \approx K^A_n(y) - K^A_m(y)
        \approx K^A_{n,m}(y\mid y).
    \end{align*}

    In both cases, by Proposition \ref{prop:conditionalBoundsFrostmanIntervals}, and the symmetry of information,
    \begin{align*}
        K^A_{n,m}(y\mid y) - \alpha(n-m) + \eps (n-m) &\gtrapprox K^{A,x}_{n,n,m}(y\mid |x-y|, y)\\
        &\approx K^{A,x}_{n,m}(y\mid  y) - K^{A,x}_{n,m}(|x-y|\mid  y)\\
        &\gtrapprox K^{A,x}_{n,m}(y\mid  y) - K^{A,x}_{n,m}(|x-y|\mid  |x-y|)\\
        &\approx K^{A}_{n,m}(y\mid  y) - K^{A,x}_{n,m}(|x-y|\mid  |x-y|).
    \end{align*}
    
    If $[m,n]$ is $(\sigma, \eta)$-Frostman with $\sigma \geq d/2$, the above inequality holds with $\alpha = 1$, proving \eqref{eq:boundSigmaGreaterDover2}.  Similarly, if $[m,n]$ is $(\sigma, \eta)$-Frostman and $(\sigma, \eta)$-Katz-Tao with $\sigma < d/2$, then the above inequality holds with $$\alpha = \min\{\sigma, \frac{d/2 + \sigma -(d-2)}{2},1\}.$$In particular, since $d = 3, 4$, we conclude that \eqref{eq:boundSigmaLessThanDover2} holds. 
 
\end{proof}

\begin{remark}
    Corollary \ref{cor:conditionalBoundsExactFrostman} can be generalized to hold for arbitrary $d$, with a nearly identical proof. We won't use this generalization in the paper, but it may be of independent interest.
    \begin{corollary}\label{cor:conditionalBoundsExactFrostmanGeneralized}
 Let $m \leq n$ such that $n\leq 2m$ and $n-m\ge C \eta^{-1} \log n$. If $[m,n]$ is $(\sigma, \eta)$-Frostman with $\sigma \geq d/2$, then
    \begin{equation*}
        K^{A,x}_{n,m}(|x-y|\mid |x-y|) \gtrapprox n-m - \eps (n-m).
    \end{equation*}
    If $[m,n]$ is $(\sigma, \eta)$-Frostman and $(\sigma, \eta)$-Katz-Tao with $\sigma < d/2$, 
    then
    \begin{equation*}
        K^{A,x}_{n,m}(|x-y|\mid  |x-y|) \gtrapprox \min\{\sigma, \frac{s + \sigma -(d-2)}{2}, 1\}(n-m) - \eps (n-m)
    \end{equation*}
\end{corollary}
\end{remark}

We now turn to the proof of Theorem \ref{thm:maintheoremEff}. The plan is as follows. We will use the partition of Section \ref{sec:LipschitzFunction} to partition the complexity function $K^A(y)$ into intervals, and apply Corollary \ref{cor:conditionalBoundsExactFrostman} to each interval. We then use the symmetry of information to sum the contribution of each part. 

We now fix an oracle $A$, $x, y\in\R^d$ and $e = \frac{y-x}{|x-y|}$ which satisfy 
    \begin{itemize}
\item[\textup{(C1)}] $\dim_H^A(x), \dim^A_H(y), \dim^{A,x}_H(e) > \frac{d}{2}$;
\item[\textup{(C2)}] $K^{A, x}_n(y) \gtrapprox K^{A}_n(y)$ for every $n\in\N$.
\end{itemize}
We fix $n\in\N$. Let $D$ be the oracle of Lemma \ref{lem:oracles} with $x$ replaced by $y$, the $\eta$ of the lemma replaced with $d/2$. By Lemma \ref{lem:oracles}(i), $K^{A,D}_n(y) \approx dn/2$ and $$K^{A,D}_m(y) = \min\{dn/2, K^A_m(y)\} + O(\log n),$$for every $m\leq n$. Let $f:[0,n] \rightarrow \R_+$ be a monotone, $(d+1)$-Lipschitz, piecewise linear function approximating the complexity of $y$ as described in Section \ref{ssec:prelimAlgMethods1}.  We choose $f$ such that $f(0) = 0$, $f(n) = dn/2$, $f(t) \geq dt/2$ for all $0 \leq t \leq n$, and $$\left| f(m) - K^{A,D}_m(y)
\right| \leq C\log n$$ for all $m \leq n$. We note that $f$ is $d/2$-Katz-Tao. 

We now prove Theorem \ref{thm:maintheoremEff} using the selection of intervals given by Proposition \ref{prop:existenceNicePartition}. 
\begin{proof}[Proof of Theorem \ref{thm:maintheoremEff}]
    Let $\eps > 0$. Let $n$ be sufficiently large. Let $G \cup R \cup B$ be the sets of intervals given by Proposition \ref{prop:existenceNicePartition}. By changing $\eta$ by a multiplicative constant (if necessary), without loss of generality we can assume that the endpoints of these intervals are integers. This gives us a partition $P$ of $[0,n]$ into intervals with integer endpoints. Since $G\cup R\cup B$ has $\lesssim \log n$ many intervals, and $P$ has at most twice more, therefore by applying the symmetry of information $\lesssim \log n$ many times, we see that 
    \begin{equation}\label{k}
        K^{A,D,x}_{n}(|x-y|) \geq \sum_{[a,b]\in P} K^{A,D,x}_{b,a}(|x-y|\mid |x-y|) - O(\log^2 n).
    \end{equation}
 We choose $n$ large enough so that $\log^2 n\le \eps n$.   
    
    By Corollary \ref{cor:conditionalBoundsExactFrostman}, which is also true for oracle $(A,D)$, if $[a,b]\in G\cup R$, then 
    \begin{equation*}\label{eq:mainThmBoudFE}
        K^{A,D,x}_{b,a}(|x-y|\mid |x-y|) \gtrapprox b-a - \eps(b-a), 
    \end{equation*}
    and if $[a, b] \in B$, then
    \begin{equation*}\label{eq:mainThmBoudFE2}
        K^{A,D,x}_{b,a}(|x-y|\mid |x-y|) \gtrapprox K^{A,D}_{b, a}(y\mid y)/2 - \eps(b-a).
    \end{equation*}
For the intervals $[a,b]\not\in G\cup R\cup B$ we use the facts that their total length is $O(\eps n)$, and there are $\lesssim\log n$ many of them. Therefore, summing for all intervals in $P$, we can see from \eqref{k} that

    \begin{align}\label{hm}
        K^{A,D,x}_{n}(|x-y|) &\geq |R| + |G|
        + \sum_{[a,b]\in B} K^{A,D}_{b, a}(y\mid y)/2 - O(\eps n)\\
        &\ge |R|+|G|+ \sum_{[a,b]\in B} (f(b)-f(a))/2 - O(\eps n),\label{mmm}
    \end{align}
     where $|R|$ and $|G|$ denotes the total length of the intervals in $R$ and $G$, respectively. 
     We will now use the trivial identity $$dn/2=f(n)-f(0)=\sum_{[a,b]\in P} (f(b)-f(a)).$$ 
In this sum, if $[a,b]\in G$ then it is $d/2$-Frostman and $d/2$-Katz Tao, therefore $f(b)-f(a)=d(b-a)/2$. If $[a,b]\in R$ then we use Remark \ref{remark:complexityfunction2}, and summing for those intervals which are not in $G\cup R\cup B$ we get at most a $O(\eps n)$ since $f$ is Lipschitz. This shows that 
$$dn/2\le d|G|/2+d|R|+\sum_{[a,b]\in B} (f(b)-f(a))+O(\eps n),$$ and, recalling also $|G|\ge |R|$ from Proposition \ref{prop:existenceNicePartition}, now $\eqref{hm}$ gives
    \begin{align*}
        K^{A,D,x}_{n}(|x-y|) &\geq |R| + |G| +dn/4-d|G|/4-d|R|/2-O(\eps n)\\
        &=dn/4+(1-d/2)|R|+(1-d/4)|G|-O(\eps n)\\
        &\ge dn/4-(3d/4-2)|R|-O(\eps n).
    \end{align*}

    If $|R| \le n/3 - O(\eps n)$, then this shows that $K^{A,D,x}_n(|x-y|) \gtrapprox 2n/3 - O(\eps n)$, and if $|R| \ge n/3 - O(\eps n)$, then the same inequality follows from \eqref{mmm}.

    Since $K^{A,x}_n(|x-y|) \gtrapprox K^{A,D,x}_{n}(|x-y|)\ge 2n/3-O(\eps n)$, the conclusion follows.
\end{proof}

\bibliographystyle{plainurl}
\bibliography{distance}

\end{document}